\documentclass[11pt]{amsart}
\pdfoutput=1

\usepackage[T1]{fontenc}
\usepackage[utf8]{inputenc}
\usepackage{lmodern}
\usepackage{libertine}
\usepackage{microtype}
\usepackage{xcolor}
\usepackage[backend=biber, maxcitenames=2, maxbibnames=99]{biblatex}
\AtEveryBibitem{%
  \clearfield{urlday}\clearfield{urlmonth}\clearfield{urlyear}%
  \clearfield{pubstate}\clearfield{eventtitle}%
  \iffieldundef{doi}{}{\clearfield{url}\clearfield{eprint}%
    \clearfield{eprinttype}\clearfield{eprintclass}}%
}
\usepackage[a4paper,width=16cm,left=2.5cm,top=3cm,bottom=3cm]{geometry}
\usepackage{enumitem}
\usepackage{amsmath, amsthm, amssymb, amsfonts, amsopn}
\usepackage{mathtools}
\usepackage{mathrsfs}
\usepackage{graphicx}
\usepackage{booktabs}
\usepackage{siunitx}
\usepackage{pifont}
\usepackage{tikz}
\usetikzlibrary{arrows.meta,positioning,calc,decorations.pathreplacing}
\newcommand{\cmark}{\ding{51}}
\newcommand{\xmark}{\ding{55}}

\theoremstyle{plain}
\newtheorem{thm}{Theorem}[section]

\newtheorem{lem}[thm]{Lemma}
\newtheorem{prop}[thm]{Proposition}

\theoremstyle{remark}
\newtheorem{rmk}[thm]{Remark}

\theoremstyle{definition}
\newtheorem{defn}[thm]{Definition}
\newtheorem{ass}{Assumption}

\newtheorem{ex}[thm]{Example}

\DeclareMathOperator{\im}{im}
\DeclareMathOperator{\col}{col}
\DeclareMathOperator{\spann}{span}
\newcommand{\cspan}{\mathop{\overline{\spann}}\nolimits}
\usepackage[foot]{amsaddr}
\usepackage[pdftex, colorlinks=true, bookmarksopen=true]{hyperref}
\usepackage{orcidlink}
\hypersetup{
  pdftitle={Data-Driven Control in Infinite-Dimensional Spaces: Fundamental Lemma and Applications},
  pdfauthor={Daniel L\'opez-Montero}
}

\def\RR {{\mathfrak R}}
\def\CC {{\mathbb C}}
\def\N {{\mathbb N}}
\def\R {{\mathbb R}}

\def\d{{\rm d}}
\def\e{{\rm e}}

\newcommand{\qand}{{\quad\text{and}\quad}}
\newcommand{\angl}[2]{\langle #1,#2\rangle}
\newcommand{\DD}{\mathscr D}

\def\eps{\varepsilon}

\title{Data-Driven Control in Infinite-Dimensional Spaces: Fundamental Lemma and Applications}
\author[D. L\'opez-Montero]{Daniel L\'opez-Montero\textsuperscript{1}~\orcidlink{0009-0004-3565-3615}}
\address{\textsuperscript{1}Chair for Dynamics, Control, Machine Learning, and Numerics (Alexander von Humboldt Professorship), Department of Mathematics, Friedrich-Alexander-Universit\"at Erlangen-N\"urnberg, 91058 Erlangen, Germany}

\date{\today}

\begin{document}

\begin{abstract}
Data-driven control of dynamical systems has attracted significant research interest. This paper establishes a theoretical foundation for data-driven control in infinite-dimensional spaces by extending the notions of \emph{persistency of excitation}, \emph{data informativity}, and Willems' fundamental lemma. These extensions enable their direct application to systems governed by partial differential equations and delay differential equations.
\end{abstract}


\maketitle

\section{Introduction}

In system identification and control, the design of the probing input is critical for ensuring that the underlying dynamics and properties are identifiable. This motivated the classical notion of \emph{persistently exciting} inputs: signals that excite all modes of the system~\parencite{astromNumericalIdentificationLinear1966, ljungSystemIdentification1998}. Willems' fundamental lemma made this precise for discrete-time linear systems, showing that a persistently exciting trajectory parametrizes the entire system behavior~\parencite{willemsNotePersistencyExcitation2005}.
Continuous-time counterparts have appeared only recently for
finite-dimensional linear systems~\parencite{rapisardaPersistencyExcitationCondition2023, schmitzContinuoustimeFundamentalLemma2024, wakaikiDataDrivenControlContinuousTime2026, schmitzDatadrivenContinuoustimeOptimal2026}.
The extension to infinite-dimensional systems, such as \emph{linear
partial differential equations (PDEs)} and \emph{delay differential equations
(DDEs)}, has by contrast received little attention; see~\parencite{kergusDataDrivenControlInfinite2021,
goseaLoewnerDatadrivenControl2021,wakaikiDataInformativityStabilization2026, pillaiBehavioralApproachControl1999}. 
In that setting, a standard state-space realization takes
the form
\begin{equation}\label{eq:hilbert-control-formal}
  \begin{aligned}
    \dot x(t)&=Ax(t)+Bu(t), & x(0)&=x_0,\\
    y(t)&=Cx(t),
  \end{aligned}
\end{equation}
where $X$, $U$, $Y$ are Hilbert spaces, $x(t)\in X$, $u(t)\in U$,
$y(t)\in Y$, the operator $A$ is the infinitesimal generator of a strongly continuous
semigroup $S=(S(t))_{t\ge0}$, and $B$ and $C$
may be unbounded on $X$. Following \textcite{pritchardLinearQuadraticControl1987}, we accommodate
them through a triple of Hilbert spaces
\begin{equation}\label{eq:intro-WXV}
  W \hookrightarrow X \hookrightarrow V,
\end{equation}
with continuous, dense injections, so that $B$ becomes bounded into the coarser space,
$B\in\mathcal L(U,V)$, and $C$ bounded on the finer one,
$C\in\mathcal L(W,Y)$. We also assume that $S(t)$ restricts to a strongly
continuous semigroup on $W$ and extends to one on $V$. Throughout, the operators
$A$, $B$, and $C$ are assumed to be \emph{unknown}. 

The following three
examples, introduced by \textcite{pritchardLinearQuadraticControl1987}, motivate the theory.

\begin{ex}[Heat equation, Neumann control]\label{ex:intro-pde}
  On the interval $(0,1)$, consider
  \[
    \begin{cases}
      \partial_t x(t,\xi)=\nu\,\partial_{\xi\xi}x(t,\xi),
        & \partial_\xi x(t,0)=u(t),\quad \partial_\xi x(t,1)=0,\\
      y(t)=x(t,\xi_0),
    \end{cases}
  \]
  with diffusivity $\nu>0$ and state space $X=L^2(0,1)$. Control enters through a Neumann
  boundary condition and the output samples the state at a point
  $\xi_0\in(0,1)$, so both $B$ and $C$ are unbounded on $X$.
\end{ex}

\begin{ex}[Wave equation, Dirichlet control]\label{ex:intro-wave}
  On the interval $(0,1)$, consider
  \[
    \begin{cases}
      \partial_{tt}w(t,\xi)=c^2\,\partial_{\xi\xi}w(t,\xi),
        & w(t,0)=u(t),\quad w(t,1)=0,\\
      y(t)=\displaystyle\int_0^1 q(\xi)w(t,\xi)\,\d\xi,
    \end{cases}
  \]
  with wave speed $c>0$. It can be written as a
  first-order system \eqref{eq:hilbert-control-formal} for
  $x=(w,\partial_t w)$ on
  $X=L^2(0,1)\times H^{-1}(0,1)$. Dirichlet boundary control makes $B$
  unbounded on $X$, while the weighted integral observation defined by $q\in L^2(0,1)$ is bounded on $X$.
\end{ex}

\begin{ex}[Delay system with delayed output]\label{ex:intro-delay}
  For a delay $h>0$, let $x_t(\tau):=x(t+\tau)$,
  $\tau\in[-h,0]$, and consider the delay system
  \[
    \begin{cases}
      \displaystyle\frac{\d}{\d t}\bigl(x(t)-Mx_t\bigr)=Lx_t+B_0u(t),\\
      y(t)=Cx_t,
    \end{cases}
  \]
  where $x(t)\in\R^n$ and $B_0\in\R^{n\times m}$, while $L$, $M$, and $C$ are
  bounded linear operators mapping $C([-h,0];\R^n)$ into $\R^n$, $\R^n$, and
  $\R^p$ respectively. We assume that the measure representing $M$ has no
  mass at $0$. Its state
  $(x(t)-Mx_t,x_t)$ belongs to
  $X=\R^n\times L^2(-h,0;\R^n)$.
  The control operator is bounded, whereas a discrete output delay generally
  makes $C$ unbounded on $X$.
\end{ex}

\paragraph{Contributions.} In this paper, we lay the foundations for data-driven control of infinite-dimensional linear systems:
\begin{enumerate}[label=\textup{(\roman*)},leftmargin=*]
    \item \emph{Informativity and persistency of excitation (PE).} We extend
    persistency of excitation and data informativity to infinite dimensions,
    and prove that PE of infinite order is necessary, but---unlike in finite
    dimensions---not sufficient, for an input to be universally informative. To
    close the gap we introduce \emph{harmonic persistency of excitation}, a
    frequency-domain condition that does imply informativity for analytic
    semigroups.

    \item \emph{Continuous-time and infinite-dimensional fundamental lemma.} We extend Willems' fundamental lemma to semigroup dynamics, with applications to linear PDEs and DDEs.
    \item \emph{Data-driven controllability test.} We give tests for approximate controllability of an infinite-dimensional system from a single trajectory.
   
    \item \emph{Data-driven linear--quadratic regulator (LQR).}
    Building on the linear--quadratic theory for infinite-dimensional
    systems~\parencite{pritchardLinearQuadraticControl1987} and on data-driven
    LQR~\parencite{schmitzContinuoustimeFundamentalLemma2024}, we derive an
    exact data-based representation and discretization.
\end{enumerate}

Figure~\ref{fig:roadmap} summarizes the logical structure of the paper.
Section~\ref{sec:setting} fixes the standing assumptions. The excitation conditions of
Section~\ref{sec:informativity} then supply the informativity hypotheses of the
fundamental lemma of Section~\ref{sec:fundamental-lemma}, on which the two
applications of Sections~\ref{sec:controllability} and~\ref{sec:lqr} rest.
Section~\ref{sec:numerics} reports the numerical experiments.

\begin{figure}[!ht]
  \centering
%
\begin{tikzpicture}[
  font=\scriptsize,
  bx/.style   ={draw,rounded corners=2pt,align=center,
                inner xsep=3pt,inner ysep=2pt,
                text width=30mm,minimum height=8.5mm,line width=0.5pt},
  main/.style ={bx,fill=black!5},
  pe/.style   ={bx,dashed,fill=none},
  gain/.style ={align=center,inner xsep=3pt,inner ysep=2pt,text width=40mm,
                anchor=north,font=\scriptsize},
  im/.style   ={-{Latex[length=1.7mm,width=1.5mm]},line width=0.6pt},
  eq/.style   ={{Latex[length=1.7mm,width=1.5mm]}-{Latex[length=1.7mm,width=1.5mm]},
                line width=0.6pt},
  lbl/.style  ={font=\tiny,inner sep=1.6pt},
]
\hyphenpenalty=10000\exhyphenpenalty=10000\relax

\node[pe]   (hpe)  at (0  , 1.6) {$\bar u$ harmonically PE};
\node[pe]   (pein) at (4.6, 1.6) {$\bar u$ PE of\\infinite order};

\node[main] (ac)   at (0  , 0  ) {$(A,B)$ approx.\\controllable};
\node[main] (inf)  at (4.6, 0  ) {$(\bar u,\bar x)$ informative
};
\node[main] (swi)  at (9.2, 0  ) {$(\bar u,\bar x)$ window\\informative};

\node[main,text width=76mm] (wfl) at (6.9,-1.55)
      {Willems' fundamental lemma};

\node[gain] (fl)   at (4.6,-2.475)
      {Data-driven LQR\\(Thms~\ref{thm:lqr-data} and~\ref{thm:lqr-data-io})};
\node[gain] (ioc)  at (9.2,-2.475)
      {Data-driven controllability\\(Thms~\ref{thm:data-fattorini-hautus}
       and~\ref{thm:io-window-controllability})};

\draw[im] (hpe) -- node[lbl,above]{Lem~\ref{lem:harmonic-pe-implies-pe}} (pein);
\draw[im,dashed] (pein) -- (inf)
      node[pos=0.5,fill=white,inner sep=1pt,font=\normalsize]{\xmark}
      node[lbl,pos=0.5,right=4pt]{Prop~\ref{prop:infinite-pe-insufficient}};
\draw[im,dashed] ($(hpe.south east)-(0.15,0)$) -- ($(inf.north west)+(0.15,0)$)
      node[pos=0.5,fill=white,inner sep=1pt,font=\normalsize]{\cmark};

\draw[eq] (ac)  -- node[lbl,above]{Thm~\ref{thm:analytic-universal-sufficiency}} (inf);
\draw[eq] (inf.east)
       -- node[lbl,above]{Thm~\ref{thm:analytic-universal-sufficiency}}
          (swi.west);

\draw[im] (inf.south) -- node[lbl,left]{Thm~\ref{thm:willems-gramian}}
          (inf.south |- wfl.north);
\draw[im] (swi.south) -- node[lbl,left]{Thm~\ref{thm:windowed-io-fundamental-lemma}}
          (swi.south |- wfl.north);

\draw[im] (fl.north  |- wfl.south) -- (fl.north);
\draw[im] (ioc.north |- wfl.south) -- (ioc.north);

\end{tikzpicture}
  \caption{Logical structure of the paper.}
  \label{fig:roadmap}
\end{figure}
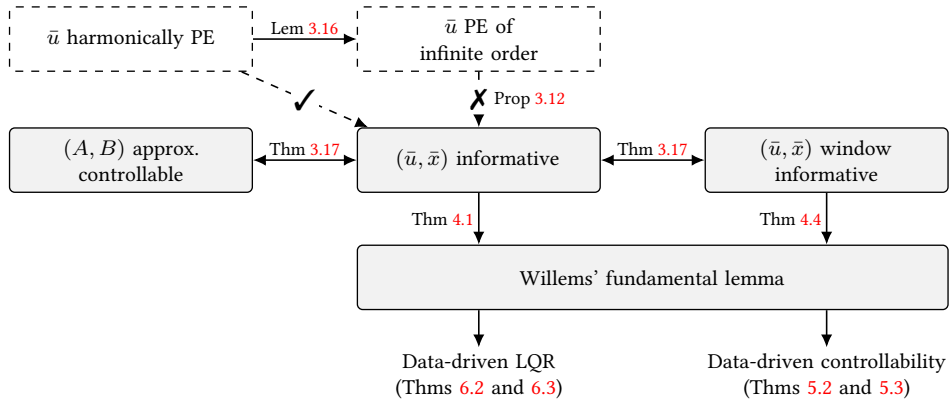

\paragraph{Notation.}
Fix $T>0$ and set $I:=[0,T]$.
All Hilbert spaces are real unless explicitly complexified.
For Hilbert spaces
$H$ and $K$, $\mathcal L(H,K)$ denotes the space of bounded
linear operators from $H$ to $K$. We
write $\angl{\cdot}{\cdot}_H$ and $\|\cdot\|_H$ for the inner product and norm
of $H$, and $\im L$ and $\ker L$ for the range and kernel of an operator $L$.
For a real Hilbert space $H$, its complexification is denoted by $H_{\CC}$ and
$\angl{\cdot}{\cdot}_H$ denotes the complex-bilinear extension of the real
inner product.
We identify $X$ with its dual and denote by $W'$ and $V'$ the duals of $W$
and $V$ with respect to the pivot space $X$, so that
$V'\hookrightarrow X\hookrightarrow W'$ continuously and densely. The
duality pairing between a space $E$ and its pivot dual $E'$ is denoted by
$\angl{\cdot}{\cdot}_{E,E'}$.
Accordingly, $B^\star\in\mathcal L(V',U)$ and
$C^\star\in\mathcal L(Y,W')$ denote the adjoints of $B$ and $C$.
The space of smooth test functions compactly supported in the interior of $I$
is denoted by $\DD(I)$, its distributional dual by $\DD'(I)$, and, for a Hilbert
space $E$, the space of $E$-valued distributions by $\DD'(I;E)$. We write $D$
for the distributional time derivative.
For a signal $z\in L^2(I;H)$, we write $\cspan z$ for the smallest closed
subspace of $H$ containing $z(t)$ for a.e.\ $t\in I$.

\section{Setting and standing assumptions}
\label{sec:setting}

The following assumptions, adopted from
\textcite{pritchardLinearQuadraticControl1987}, specify the admissibility of
$B$ and $C$ and the compatibility of the state spaces.

\begin{ass}
  \label{hyp:H1}
  There exists $b>0$ such that, for every $u\in L^2(0,T;U)$,
  \[
    \int_0^T S(T-r)Bu(r)\,\d r\in W
    \qand
    \Big\|\int_0^T S(T-r)Bu(r)\,\d r\Big\|_{W}\le b\,\|u\|_{L^2(0,T;U)}\,.
  \]
\end{ass}

\begin{ass}
  \label{hyp:H2}
  There exists $c>0$ such that, for every $x\in W$,
  \[
    \|CS(\cdot)x\|_{L^2(0,T;Y)}\le c\,\|x\|_{V}.
  \]
\end{ass}

\begin{ass}\label{hyp:compat}
  The domain $Z:=\mathcal D_V(A)$ of the generator on $V$ embeds continuously
  and densely into $W$, where $Z$ carries the graph norm.
\end{ass}

Assumption~\ref{hyp:H1} ensures that the control convolution takes values in
$W$. Thus, for $x_0\in X$ and $u\in L^2(I;U)$, the system
\eqref{eq:hilbert-control-formal} has a \emph{mild solution}
$x\in C(I;X)$ given by
\begin{subequations}\label{eq:hilbert-mild}
\begin{align}
x(t;x_0,u)
&= S(t)x_0 + L_tu,
\qquad
L_tu := \int_0^t S(t-r)Bu(r)\,\d r,
\label{eq:hilbert-mild-state}\\
\intertext{If, in addition, $x_0\in W$, then $x\in C(I;W)$ and the
associated output is}
y(t)
&= CS(t)x_0 + CL_tu.
\label{eq:hilbert-mild-output}
\end{align}
\end{subequations}
In this case, $y\in C(I;Y)$.

Assumption~\ref{hyp:H2} also extends the free output $CS(\cdot)x_0$ by density
from $W$ to $V$, and hence in particular to $X$, as an element of $L^2(I;Y)$.
Together with the forced response, this yields the two solution maps used
throughout---the input-to-output map $\mathcal F_\tau$ and the
\emph{observability operator} $\mathcal O_\tau$---and the shift identity
relating a record to its windows.

\begin{lem}\label{lem:finite-horizon-output}
  Under Assumptions~\ref{hyp:H1} and~\ref{hyp:H2}, for every $\tau>0$, the operators
  \[
    \begin{aligned}
      \mathcal F_\tau: L^2(0,\tau;U)&\to L^2(0,\tau;Y),
      &\qquad
      \mathcal O_\tau: X&\to L^2(0,\tau;Y),\\
      u&\mapsto CL_{(\cdot)}u,
      &
      x_0&\mapsto CS(\cdot)x_0,
    \end{aligned}
  \]
  are well-defined, linear and bounded.
  The output generated by $(x_0,u)$ is
  $y=\mathcal O_\tau x_0+\mathcal F_\tau u$.
  Moreover, for $0\le s\le s+\rho\le\tau$, one has
  $y(s+\cdot)=\mathcal O_\rho x(s)+\mathcal F_\rho u(s+\cdot)$ in
  $L^2(0,\rho;Y)$.
\end{lem}

The proof is given in Appendix~\ref{app:finite-horizon-output}. The
observability operator maps an initial state to the unforced part of the output.

In expressions involving $\rho(A)$ or $\mathcal D(A^\star)$, $A$ denotes
the generator on $X$. Assumption~\ref{hyp:H1} induces a compatible control
operator $B_X\in\mathcal L(U,X_{-1})$, where $X_{-1}$ is the extrapolation
space of $X$, with the same control convolution. We also write
$B^\star:\mathcal D(A^\star)\to U$ for its adjoint; it agrees with the
previously defined adjoint on their common domain
$\mathcal D(A^\star)\cap V'$. The construction is given in
Appendix~\ref{app:control-realization}.
Writing $A_V$ and $A_{X,-1}$ for the generators on $V$ and $X_{-1}$,
respectively, the compatible realizations satisfy, for large real $r$,
\[
  Q_r:=(r\operatorname{id}-A_V)^{-1}B
      =(r\operatorname{id}-A_{X,-1})^{-1}B_X\in\mathcal L(U,W),
  \qquad \|Q_r\|_{\mathcal L(U,W)}=O(r^{-1/2}).
\]
This decay follows from Assumption~\ref{hyp:H1}; the estimate
\eqref{eq:control-resolvent-decay} in the appendix supplies the details.

\subsection{Controllability and observability}

In the infinite-dimensional setting, controllability and observability each admit two distinct formulations, \emph{exact} and \emph{approximate}~\parencite{tucsnakObservationControlOperator2009,boyerControllabilityLinearParabolic2020,trelatControlFiniteInfinite2024}.

\begin{defn}[Controllability]\label{def:controllability}
  For $x_0\in X$, let $\RR_T(x_0):=\{x(T;x_0,u)\;:\;u\in L^2(I;U)\}$ denote the
  set of states reachable at time $T$ from $x_0$, so that $\RR_T(0)=\im L_T$. The pair $(A,B)$ is
  \emph{exactly controllable in time $T$} if $\RR_T(0)=X$, and
  \emph{approximately controllable in time $T$} if
  $\overline{\RR_T(0)}^{\,X}=X$.
\end{defn}

\begin{defn}[Observability]\label{def:observability}
  The pair $(A,C)$ is \emph{approximately observable in time $T$} if
  $\ker\mathcal O_T=\{0\}$, and \emph{exactly observable in time $T$} if,
  for some $c_T>0$,
  \begin{equation*}
    \|\mathcal O_Tx\|_{L^2(0,T;Y)}\ge c_T\|x\|_X,
    \qquad x\in X.
  \end{equation*}
\end{defn}

\section{Informativity and persistency of excitation}
\label{sec:informativity}

Everything that follows rests on a single measured record being rich enough to
stand in for the model. This section makes that requirement precise and then
addresses the design question it raises: \emph{how should the probing input be
chosen so that the resulting record is informative?}

\subsection{Informative data}
\label{subsec:informativity}

The notion of \emph{data informativity} has generalized Willems'
fundamental lemma into a framework for quantifying the information content of a
dataset---covering identifiability, controllability, stabilizability, and
dissipativity, among others~\parencite{vanwaardeDataInformativityNew2020,
eisingDataInformativityObservability2020, steentjesDataDrivenControlInformativity2022,
vanwaardeInformativityApproachDataDriven2023}. We formulate informativity as the
requirement that the measured signal explore every direction of the signal space.

\begin{defn}[Informative data]\label{def:hilbert-informative}
  Let $H$ be a Hilbert space embedded continuously and densely in $X$. A pair
  $(\bar u,\bar x)\in L^2(I;U\times H)$ is
  \emph{informative in $U\times H$} if
  \begin{equation}\label{eq:informative-span}
    \cspan(\bar u,\bar x)=U\times H.
  \end{equation}
  Equivalently, if $(\alpha,\beta)\in U\times H$ satisfies
  $\angl{\alpha}{\bar u(t)}_U+\angl{\beta}{\bar x(t)}_H=0$ for a.e.\ $t\in I$,
  then $(\alpha,\beta)=0$.
\end{defn}

Condition \eqref{eq:informative-span} has an equivalent operator-theoretic
form, which is the one we work with throughout. For $z\in L^2(I;H)$, define
\begin{equation}\label{eq:hilbert-gramian}
  \mathcal G(z)
  :=
  \int_0^T z(t)\otimes z(t)\,\d t
  \ \in\mathcal L(H),
\end{equation}
where $a\otimes b$ denotes the rank-one operator
$\zeta\mapsto\angl{b}{\zeta}_H\,a$ on $H$. Then $\mathcal G(z)$ is
self-adjoint, nonnegative, and trace-class, and its kernel collects the vectors
orthogonal to $z(t)$ for a.e.\ $t\in I$,
\begin{equation}\label{eq:gramian-kernel}
  \ker\mathcal G(z)=(\cspan z)^{\perp}.
\end{equation}
Applied on the product space $U\times H$ to $z=(\bar u,\bar x)$, this
characterizes informativity.
\begin{lem}[Gramian characterization of informative data]\label{lem:informative-gramian}
  A pair $(\bar u,\bar x)\in L^2(I;U\times H)$ is informative in $U\times H$ if
  and only if $\mathcal G(\bar u,\bar x)$ is injective.
\end{lem}

The proof is given in Appendix~\ref{app:informative-gramian}.

Unless a space is specified, \emph{informative} means informative in
$U\times X$. This is the relevant notion when no
output is recorded, or when $C$ is bounded. Otherwise, we
use $U\times W$, the reason being that, by \eqref{eq:hilbert-mild-output}, the
output is defined pointwise only for states in $W$.
Since $W\hookrightarrow X$ continuously and densely, informativity in $U\times W$ implies informativity in $U\times X$; the converse fails in general.

\begin{rmk}\label{rmk:gramian-not-coercive}
When $U\times X$ is finite-dimensional, injectivity is equivalent to coercivity, $\mathcal G(\bar u,\bar x)\succ0$, recovering the Gramian condition of the finite-dimensional theory for discrete-time \parencite{vanwaardeDataInformativityNew2020} and continuous-time systems \parencite{schmitzContinuoustimeFundamentalLemma2024}.
\end{rmk}

\subsection{Window informativity}
\label{subsec:window-informativity}

Data informativity imposes a pointwise constraint: the values
$(\bar u,\bar x)(t)$ must span a dense subspace of $U\times X$. When only
outputs are measured this is not enough, because infinite-dimensional states
generally cannot be recovered pointwise; the state information is instead
carried by an entire input--output window. The trajectories that such a record
makes available are the superpositions of its shifted windows, and these must
now be total in $L^2(0,T;U)\times X$.

\begin{defn}[Window synthesis and informativity]\label{def:window-informative}
  Let $(\bar u,\bar x)$ be an input--state record on $[0,T+T']$. For
  $\phi\in L^2(0,T')$, define the \emph{window synthesis operator}
  \[
    \mathcal Z_{T,T'}^{\bar u,\bar x}\phi
      :=\int_0^{T'}\phi(s)(\bar u(s+\cdot),\bar x(s))\,\d s
      \ \in L^2(0,T;U)\times X.
  \]
  The record $(\bar u,\bar x)$ is \emph{window informative} if 
  $\overline{\im\mathcal Z_{T,T'}^{\bar u,\bar x}}=L^2(0,T;U)\times X$.
\end{defn}

Because it requires the shifted windows to be total rather than just constraining the pointwise values, window informativity is a stronger notion than data informativity.

\begin{lem}
  \label{lem:window-implies-informative}
  Suppose that Assumption~\ref{hyp:H1} holds, and let $(\bar u,\bar x)$ be the
  record on $[0,T+T']$ generated by $(\bar x_0,\bar u)\in X\times
  L^2(0,T+T';U)$. If $(\bar u,\bar x)$ is window informative, then it is
  informative on $(0,T+T')$.
\end{lem}

The proof is given in Appendix~\ref{app:window-implies-informative}.

The converse fails, because window informativity
requires the shifted inputs to be total in the
infinite-dimensional space $L^2(0,T;U)$, which no condition posed pointwise in
$U\times X$ can deliver.

\begin{ex}\label{ex:window-strictly-stronger}
  Let $U=X=\ell^2(\N)$ with orthonormal basis $(e_n)_{n\ge1}$, and let $A=0$ and
  $B=\operatorname{id}$, so that $(A,B)$ is exactly controllable in every time.
  Take $\bar x_0=0$ and
  \begin{equation}\label{eq:window-strictly-stronger-input}
    \bar u(t):=\sum_{n\ge1}2^{-n}\sin(nt)\,e_n .
  \end{equation}
  Then $(\bar u,\bar x)$ is informative on every interval, but
  it is window informative for no $T,T'>0$.
\end{ex}

The verification is given in
Appendix~\ref{app:window-strictly-stronger}. Theorem~\ref{thm:analytic-universal-sufficiency}
below shows that the two notions coincide for analytic semigroups driven by a
harmonic input.

\subsection{Persistently exciting inputs}
\label{subsec:classical-pe}

Having made informativity precise, we return to the design question
posed at the start of this section.

In the finite-dimensional setting, the classical answer is that the input must be
\emph{persistently exciting} of sufficiently high order, that is, rich enough to
excite all modes of the system. In discrete time, this amounts to a full-rank condition~\parencite{willemsNotePersistencyExcitation2005}, while in continuous time it amounts to an independence condition on the derivatives of the input~\parencite{schmitzContinuoustimeFundamentalLemma2024}.
We first show that this classical notion, when suitably extended, remains necessary for an input to be universally informative in infinite dimensions. However, while classical excitation is sufficient in finite dimensions, it fails to guarantee informativity in infinite-dimensional systems. This limitation motivates a new stronger condition which we call \emph{harmonic persistency}.

To formalize the continuous-time classical condition, we first define the derivative moments. For $\bar u\in L^2(I;U)$ and $L\ge1$, define
\[
  \langle D^k\bar u,\phi\rangle_{\DD',\DD}
  :=(-1)^k\int_0^T\bar u(t)\,\phi^{(k)}(t)\,\d t\in U,
  \qquad k=0,\dots,L-1,\quad \phi\in\DD(I).
\]
Placing derivatives on the test function requires no regularity of $\bar u$. 

\begin{defn}
  \label{def:PE}
  Let $\bar u\in L^2(I;U)$. We say that $\bar u$ is
  \emph{persistently exciting of order $L$} if whenever
  $\zeta_0,\dots,\zeta_{L-1}\in U$ satisfy
  \[
    \sum_{k=0}^{L-1}\angl{\zeta_k}{D^k\bar u}_U=0
    \quad\text{in }\DD'(I),
    \qquad\text{then}\qquad \zeta_0=\cdots=\zeta_{L-1}=0.
  \]
  We say that $\bar u$ is \emph{persistently exciting of infinite order} if it
  is persistently exciting of every finite order $L\ge1$.
\end{defn}

\begin{rmk}
  When the input is sufficiently regular, $\bar u\in H^{L-1}(I;U)$, one has
  $D^k\bar u=\bar u^{(k)}$ for $k\le L-1$, and persistency of excitation of order
  $L$ is equivalent to the injectivity of the Gramian
  $\mathcal G(\bar u,\dot{\bar u},\dots,\bar u^{(L-1)})$ on $U^L$. In finite dimensions
  this amounts to the classical coercivity condition
  $\mathcal G(\bar u,\dot{\bar u},\dots,\bar u^{(L-1)})\succ0$ of
  \textcite[Def.~16]{schmitzContinuoustimeFundamentalLemma2024}.
\end{rmk}

Before addressing sufficiency, we establish that persistency of excitation is
necessary for an input to be \emph{universally informative}, that is, to produce
an informative trajectory for every system in the class; following
\textcite{shakouriNewPerspectiveWillems2025,shakouriExperimentDesignUsing2025},
such an input is called \emph{universal}.

\begin{prop}\label{prop:hilbert-pe-necessary}
  Let $X$ be a separable Hilbert space, let $N$ denote the dimension of $X$,
  and let $\bar u\in L^2(I;U)$.
  Suppose that, for every approximately controllable pair $(A,B)$ on $X$
  satisfying Assumption~\ref{hyp:H1} and every initial state $x_0\in X$,
  the trajectory $\bar x=x(\cdot;x_0,\bar u)$ makes $(\bar u,\bar x)$
  informative. Then, if $N<\infty$, $\bar u$ is
  persistently exciting of order $N+1$, whereas if $N=\infty$,
  $\bar u$ is persistently exciting of infinite order.
\end{prop}

The proof is given in Appendix~\ref{app:pe-necessity}.


\begin{lem}
  \label{lem:window-pe-necessary}
  If $(\bar u,\bar x)$ is window informative on $(0,T+T')$, then $\bar u$ is
  persistently exciting of infinite order on $(0,T+T')$.
\end{lem}

The proof is given in Appendix~\ref{app:window-pe-necessity}.

\subsection{The infinite-dimensional obstacle}
\label{subsec:pe-obstacle}

In finite dimensions, the necessary condition of
Proposition~\ref{prop:hilbert-pe-necessary} is also sufficient.

\begin{prop}
  \label{prop:finite-pe-sufficient}
  Suppose that $\dim X=N<\infty$ and that $(A,B)$ is controllable.
  If $\bar u$ is persistently exciting of order $N+1$, then
  $(\bar u,\bar x)$ is informative.
\end{prop}

This is a distributional version of
\parencite[Prop.~21]{schmitzContinuoustimeFundamentalLemma2024}; the proof is given
in Appendix~\ref{app:finite-pe-sufficiency}. 
Crucially, the proof relies on
the Cayley--Hamilton theorem to close the recursion after finitely many steps.
This argument has no infinite-dimensional analogue, and the following counterexample shows that the condition is not sufficient in infinite dimensions.

\begin{prop}
  \label{prop:infinite-pe-insufficient}
  Let $X=\ell^2(\N)$ with orthonormal basis $(e_n)_{n\ge1}$, let $U=\R$, and set
  \begin{equation}\label{eq:pe-counterexample}
    A_0e_n:=-ne_n,
    \qquad
    b:=(2^{-n})_{n\ge1},
    \qquad
    Bv:=bv,
    \qquad
    A:=A_0+b\otimes b,
  \end{equation}
  with $\mathcal D(A)=\mathcal D(A_0)$. Then $(A,B)$ is approximately
  controllable and the input
  \begin{equation}\label{eq:pe-counterexample-input}
    \bar u(t):=-\sum_{n\ge1}4^{-n}e^{-nt}
  \end{equation}
  is persistently exciting of infinite order. Yet, if $\bar x$ is the trajectory
  it generates from $\bar x_0=b$, the record $(\bar u,\bar x)$ is not
  informative.
\end{prop}

The verification is given in
Appendix~\ref{app:infinite-pe-insufficient}.
The construction is a hidden feedback: $\bar u=-\angl{b}{\bar x}_X$ is the
static feedback that the rank-one perturbation $b\otimes b$ cancels, so
$\bar x$ evolves under the unperturbed $A_0$ and the record stays, for all
time, on the graph of that feedback. The obstruction is therefore a relation
\emph{between} the input and the state, which no condition on the derivatives
of $\bar u$ alone can see. In finite dimensions, Cayley--Hamilton closes the
recursion of Proposition~\ref{prop:finite-pe-sufficient} after $N$ steps and
rules such a relation out; here the recursion never closes, and the infinitely
many constraints it imposes on the annihilator remain consistent.


\subsection{Harmonic persistency of excitation}
\label{subsec:analytic-sufficient-input}
The counterexample isolates the obstruction: an
infinite-dimensional system can impose infinitely many constraints on a
candidate annihilator without ever triggering an algebraic identity theorem. We
therefore trade algebraic derivative moments for analytic continuation, using
inputs built from bounded frequencies that excite a total family of directions.

\begin{defn}[Harmonic signal]\label{def:harmonic-signal}
  A signal $h\colon\R\to U$ is \emph{harmonic} if it admits a representation
  \begin{equation*}
    h(t)=\sum_{\ell\ge1} v_\ell\,e^{i\omega_\ell t},
    \qquad v_\ell\in U_{\CC}\setminus\{0\},\quad\omega_\ell\in\R,
  \end{equation*}
  with pairwise distinct and bounded frequencies,
  $\Omega:=\sup_\ell|\omega_\ell|<\infty$, absolutely summable coefficients,
  $\sum_{\ell\ge1}\|v_\ell\|_{U_{\CC}}<\infty$, and the set of pairs
  $(\omega_\ell,v_\ell)$ invariant under $(\omega,v)\mapsto(-\omega,\overline v)$,
  with $v_\ell\in U$ when $\omega_\ell=0$.
\end{defn}

Such an $h$ is real-valued, bounded and continuous, its coefficients are
recovered uniquely by the Bohr mean $\hat h(\omega_\ell)=v_\ell$, and it is
entire of exponential type at most $\Omega$, so that identities established on
the finite window $I$ propagate to the whole line; see
Appendix~\ref{app:harmonic-pe}.

\begin{defn}[Harmonically persistently exciting]\label{def:harmonic-pe}
  A harmonic signal $h$ is \emph{harmonically persistently exciting} if there is
  a finite or countable total family $(\phi_j)_{j\in J}$ in $U_{\CC}$ such that,
  for each $j\in J$, there are pairwise distinct frequencies
  $(\omega_{j,k})_{k\ge1}$ of $h$, nonzero scalars $(a_{j,k})_{k\ge1}$, and a
  limit $\nu_j\in\R$ satisfying
  \[
    \hat h(\omega_{j,k})=a_{j,k}\,\phi_j
    \qand
    \omega_{j,k}\xrightarrow{k\to\infty}\nu_j .
  \]
\end{defn}

Such signals exist on every nonzero separable $U$ and can be built from any
orthonormal basis. The condition implies the classical one, and strictly so:
the input \eqref{eq:pe-counterexample-input} of
Proposition~\ref{prop:infinite-pe-insufficient} is persistently exciting of
infinite order but is not harmonic.

\begin{lem}
  \label{lem:harmonic-pe}
  Every nonzero separable Hilbert space $U$ admits a harmonically persistently exciting
  signal. Indeed, given a real orthonormal basis $(\phi_j)_{j\in J}$ of $U$,
  where $J=\{1,\ldots,\dim U\}$ or $J=\N$, the
  signal
  \begin{equation*}
    h(t):=\sum_{j\in J}\sum_{k\ge1}2^{-j-k}\cos(\omega_{j,k}t)\,\phi_j, \quad\text{where}\quad
    \omega_{j,k}:=2-2^{-j}-2^{-j-k},
  \end{equation*}
  is harmonically persistently exciting.
\end{lem}

\begin{lem}
  \label{lem:harmonic-pe-implies-pe}
  Let $h$ be harmonically persistently exciting and let $\sigma\in\R$. Then
  $\bar u(t):=e^{\sigma t}h(t)$, restricted to $I$, is persistently exciting of
  infinite order.
\end{lem}

The proofs are given in Appendix~\ref{app:harmonic-pe}.

\begin{thm}
  \label{thm:analytic-universal-sufficiency}
  Let $U$ and $X$ be separable real Hilbert spaces, let $A$ generate an
  analytic $C_0$-semigroup $S$ on $X$, and suppose that
  Assumptions~\ref{hyp:H1} and~\ref{hyp:compat} hold.
  Let $h$ be harmonically persistently exciting with frequencies
  $(\omega_\ell)_{\ell\ge1}$ and directions $(\phi_j)_{j\in J}$. Fix
  $T'>0$ and $\sigma\in\R$ such that
  $\sigma+i\,\overline{\{\omega_\ell:\ell\ge1\}}\subset\rho(A)$ and
  \begin{equation}\label{eq:analytic-growth-bound}
    \|S(t)\|_{\mathcal L(X)}\le M e^{\sigma t},
    \qquad t\ge0,
  \end{equation}
  for some $M\ge1$. Apply the input $\bar u(t):=e^{\sigma t}h(t)$ on
  $[0,T+T']$ and write $\bar x$ for the state trajectory with initial state
  $x_0$. The following are equivalent:
  \begin{enumerate}[label=\textup{(\roman*)},leftmargin=*]
    \item\label{item:analytic-approx-control} $(A,B)$ is approximately
      controllable in time $T$;
    \item\label{item:analytic-all-informative}
      $(\bar u,\bar x)$ is informative in $U\times W$ on $I$ for every
      $x_0\in W$;
    \item\label{item:analytic-window-all-state}
      $(\bar u,\bar x)$ is window informative for every $x_0\in X$.
  \end{enumerate}
  Moreover, \ref{item:analytic-all-informative} and
  \ref{item:analytic-window-all-state} imply
  \ref{item:analytic-approx-control} already at the single state $x_0=0$.
\end{thm}

The proof is given in Appendix~\ref{app:analytic-universal-sufficiency}.

\begin{ex}
  \label{ex:boundary-heat-informative}
  Consider the heat equation of Example~\ref{ex:intro-pde} with $\nu=1$.
  There are spaces $W\hookrightarrow X\hookrightarrow V$, measuring the
  semigroup orbit in $H^1$ near the observation point $\xi_0$, for which
  Assumptions~\ref{hyp:H1}, \ref{hyp:H2}, and~\ref{hyp:compat} hold and
  $C\in\mathcal L(W,\R)$, while $B$ and $C$ remain unbounded on $X$.
  Take $\sigma=0$ and
  \[
    \bar u(t):=\sum_{k\ge1}2^{-k}\cos(\omega_kt), \qquad \omega_k:=1+\frac{1}{k+1}.
  \]
  Then $\bar u$ is harmonically persistently exciting and $(A,B)$ is
  approximately controllable in every positive time, so
  Theorem~\ref{thm:analytic-universal-sufficiency} makes $(\bar u,\bar x)$
  informative in $U\times W$ for every initial state $\bar x_0\in W$, and
  window informative for every $\bar x_0\in X$.
  Appendix~\ref{app:boundary-heat-informative} gives the details.
\end{ex}

\section{Willems' fundamental lemma}
\label{sec:fundamental-lemma}

Willems' fundamental lemma \parencite{willemsNotePersistencyExcitation2005} shows that a single, sufficiently informative measured trajectory characterizes all admissible trajectories of a linear system, without requiring explicit knowledge of the system matrices $A$, $B$, and $C$.

In this section, we extend this result to the infinite-dimensional setting. The primary challenge lies in interpreting the state derivative. The forcing term $Bu$ poses no difficulty, since $B\in\mathcal L(U,V)$. The generator $A$, however, is unbounded: it takes values in $V$ only on the domain $\mathcal D_V(A)$ of its realization in $V$, which need not contain the space $W$ in which the record is measured. Consequently, $Ax$ may leave $V$ when $x\in W$, and the state derivative must be evaluated in a larger space.

To resolve this, we use the extrapolation space $V_{-1}$ associated with $A$ \parencite[Sec.~2.10]{tucsnakObservationControlOperator2009}: a Hilbert space into which $V$ embeds continuously and densely, and to which $S$ extends as a strongly continuous semigroup whose generator $A_{-1}\in\mathcal L(V,V_{-1})$ is the unique continuous extension of $A$. We use the same symbols $A$ and $S$ for these extensions, reserving $A_{-1}$ for arguments in which the distinction matters. Since $X\hookrightarrow V\hookrightarrow V_{-1}$, we have $A\in\mathcal L(X,V_{-1})$ and $B\in\mathcal L(U,V_{-1})$.

\subsection{Input--state--output behavior}
\label{subsec:input-state-output}

With the dynamics interpreted in $V_{-1}$, the input--state--output behavior is captured by the bounded \emph{graph operator}
\begin{equation}\label{eq:graph-operator}
  \Gamma:U\times W\to U\times W\times V_{-1}\times Y,
  \qquad
  \Gamma(v,\xi):=(v,\xi,A\xi+Bv,C\xi).
\end{equation}
Since $\|\Gamma(v,\xi)\|_{U\times W\times V_{-1}\times Y}^2\ge\|(v,\xi)\|_{U\times W}^2$, the operator is
bounded below. Hence it is a homeomorphism onto its closed range
$\im\Gamma$. The fundamental lemma below identifies this subspace from a
single informative data record.

This identification comes in two equivalent forms. The first is pointwise, relying directly on the state derivative as in \parencite{schmitzContinuoustimeFundamentalLemma2024}. However, this formulation is stated in $V_{-1}$, a space that depends on the unknown generator $A$ and is therefore hard to use in a purely data-driven setting.

To avoid differentiating the record, we provide a second, weak form, which moves the derivative onto test functions in $H_0^1(I)$ and never differentiates the record itself, making it suitable for sampled measurements. Following \textcite{wakaikiDataDrivenControlContinuousTime2026}, for any given $(u,x,y)\in L^2(I;U\times W\times Y)$, we define the associated \emph{synthesis operator} as
\[\begin{aligned} \mathcal W_{(u,x,y)} :H_0^1(I)&\to U\times W\times V_{-1}\times Y,\\ \phi &\mapsto{} \Bigl( \int_0^T\!\phi\,u\,\d t,\; \int_0^T\!\phi\,x\,\d t,\; -\int_0^T\!\dot\phi\,x\,\d t,\; \int_0^T\!\phi\,y\,\d t \Bigr). \end{aligned}\]
Its third component in fact lies in $W\hookrightarrow V_{-1}$, by the Cauchy--Schwarz inequality. If, in addition, $x\in H^1(I;V_{-1})$, integration by parts gives $-\int_0^T\dot\phi\,x\,\d t=\int_0^T\phi\,\dot x\,\d t$ as an identity in $V_{-1}$.

\begin{thm}[Willems' fundamental lemma]
  \label{thm:willems-gramian}
  Suppose that Assumption~\ref{hyp:H1} holds. Let $\bar x_0\in W$, let
  $\bar u\in L^2(I;U)$, let $(\bar x,\bar y)$ be the mild solution generated by
  $(\bar x_0,\bar u)$, and suppose that $(\bar u,\bar x)$ is informative in
  $U\times W$. Let $\mathcal G:=\mathcal G(\bar u,\bar x,\dot{\bar x},\bar y)$ be the Gramian
  \eqref{eq:hilbert-gramian} on $U\times W\times V_{-1}\times Y$. Then
  \begin{equation}\label{eq:graph-identification}
    (\ker\mathcal G)^\perp
    =\overline{\im\mathcal W_{(\bar u,\bar x,\bar y)}}^{U\times W\times V_{-1}\times Y}
    =\im\Gamma.
  \end{equation}
  Moreover, for
  every $u\in L^2(I;U)$, $x\in C(I;W)\cap H^1(I;V_{-1})$, and $y\in L^2(I;Y)$,
  the following are equivalent:
  \begin{enumerate}[label=\textup{(\roman*)},leftmargin=*]
    \item\label{item:wfl-dynamics}
    The triple $(u,x,y)$ is a trajectory of
    \begin{equation*}
    \begin{cases}
    \dot x=Ax+Bu \quad\text{in }L^2(I;V_{-1}),\\
    y=Cx \quad\text{in }L^2(I;Y);
    \end{cases}
    \end{equation*}
    \item\label{item:wfl-pointwise}
      $(u,x,\dot x,y)(t)\in(\ker\mathcal G)^\perp$ for a.e.\ $t\in I$;
    \item\label{item:wfl-synthesis}
      $\im\mathcal W_{(u,x,y)}\subseteq(\ker\mathcal G)^\perp$.
  \end{enumerate}
  Whenever these conditions hold, $(x,y)$ is the mild solution
  \eqref{eq:hilbert-mild} generated by $(x(0),u)$.
\end{thm}

The proof is given in Appendix~\ref{app:willems-gramian}.

\begin{rmk}[Input--state behavior]\label{rmk:input-state-case}
  When the output is not recorded, that is, when
  $Y=\{0\}$, the fourth component of $\Gamma$ and of $\mathcal W$ disappears
  and the state may be measured in $X$. 
\end{rmk}

\subsection{Input--output behavior}
\label{subsec:io-behavior}

When only input--output data is available, the state at the beginning of each candidate trajectory is latent. Following Section~\ref{subsec:window-informativity}, we therefore work with entire input--output windows of length $T$, which encode that latent state through the finite-horizon solution maps of Lemma~\ref{lem:finite-horizon-output}.

\begin{defn}[Input--output behavior]\label{def:io-behavior}
  For $T>0$, the \emph{input--output behavior} on $[0,T]$ is
  \[
    \mathcal B_T^{u,y}:=\bigl\{(u,y)\in L^2(0,T;U)\times L^2(0,T;Y)\ :\
      y=\mathcal O_Tx_0+\mathcal F_Tu\ \text{ for some }x_0\in X\bigr\}.
  \]
\end{defn}

Suppose we record an input--output trajectory $(\bar u,\bar y)$ on $[0,T+T']$. The shift identity $\bar y(s+\cdot)=\mathcal O_T\bar x(s)+\mathcal F_T\bar u(s+\cdot)$ of Lemma~\ref{lem:finite-horizon-output} places every recorded window $(\bar u(s+\cdot),\bar y(s+\cdot))$ in $\mathcal B_T^{u,y}$, with latent state $\bar x(s)$. Moreover, a pair in $\mathcal B_T^{u,y}$ depends linearly and boundedly on its parameters $(u,x_0)$, so weighted superpositions of the recorded windows are again trajectories. Informativity of the unmeasured state windows thus transfers to the measured input--output windows.

\begin{thm}[Fundamental lemma, input--output]
  \label{thm:windowed-io-fundamental-lemma}
  Under Assumptions~\ref{hyp:H1} and~\ref{hyp:H2}, let $(\bar u,\bar x,\bar y)$
  be the record on $[0,T+T']$ generated by $\bar u$ and an initial state
  $x_0\in X$, and let
  \begin{equation}\label{eq:io-window-synthesis}
    \mathcal Y_{T,T'}^{\bar u,\bar y}\phi
      :=\int_0^{T'}\phi(s)(\bar u(s+\cdot),\bar y(s+\cdot))\,\d s,
    \qquad \phi\in L^2(0,T').
  \end{equation}
  If $(\bar u,\bar x)$ is window informative, then
  \begin{equation}\label{eq:io-behavior-identification}
    \overline{\im\mathcal Y_{T,T'}^{\bar u,\bar y}}
      =\overline{\mathcal B_T^{u,y}} .
  \end{equation}
\end{thm}

The proof and the auxiliary parametrization (Lemma~\ref{lem:io-parametrization})
are given in Appendix~\ref{app:windowed-io-fundamental-lemma}.
The identification \eqref{eq:io-behavior-identification} requires no
observability. Approximate observability in time $T$ makes the latent state
unique; exact observability makes its recovery continuous and
$\mathcal B_T^{u,y}$ closed.

\section{Data-driven controllability}
\label{sec:controllability}

For finite-dimensional systems, controllability can be verified from a single
measured trajectory, provided that the trajectory is informative
\parencite{vanwaardeDataInformativityNew2020,mishraDataDrivenTestsControllability2021}.
The infinite-dimensional counterpart rests on the spectral characterization of
approximate controllability due to Fattorini, which we state first.

\begin{prop}[Fattorini--Hautus {\parencite[Thm.~III.3.7]{boyerControllabilityLinearParabolic2020}}]\label{prop:fattorini-hautus}
  Suppose that $B$ satisfies Assumption~\ref{hyp:H1}, that $A$ generates an
  analytic semigroup, that $A^\star$ has compact resolvent and a complete
  system of root vectors, and that $B^\star:\mathcal D(A^\star)\to U$ is
  bounded for the graph norm of $\mathcal D(A^\star)$. Then
  \begin{equation*}
    (A,B) \text{ is approximately controllable}
    \quad\Longleftrightarrow\quad
    \ker B^\star\cap\ker(A^\star-\lambda\operatorname{id}_{X_{\CC}})=\{0\}
    \quad\forall\lambda\in\CC .
  \end{equation*}
  In particular, approximate controllability is independent of $T$.
\end{prop}

The Fattorini--Hautus test is a statement about the unknown operators $A^\star$
and $B^\star$. We now convert it into a test on the measured record: an
eigenvector obstructing controllability leaves a visible trace, namely a state
projection that is a pure exponential in time. The input--state test
comes in two forms: a one-sided certificate, valid for an arbitrary record, and
an exact characterization, valid once the record is informative.

\begin{thm}[Input--state tests for approximate controllability]
  \label{thm:data-fattorini-hautus}
  Suppose that the hypotheses of Proposition~\ref{prop:fattorini-hautus} hold.
  Let $\bar x=x(\cdot;x_0,\bar u)$ be the
  mild solution generated by $(x_0,\bar u)\in X\times L^2(I;U)$, and consider
  the exponential identity
  \begin{equation}\label{eq:data-fh}
    \angl{\eta}{\bar x(t)}_X=\kappa\,e^{\lambda t}
    \qquad(t\in I).
  \end{equation}
  \begin{enumerate}[label=\textup{(\roman*)},leftmargin=*]
    \item\label{item:data-fh-certificate}
      If no $\lambda,\kappa\in\CC$ and $\eta\in X_{\CC}\setminus\{0\}$ satisfy
      \eqref{eq:data-fh}, then $(A,B)$ is approximately controllable.
    \item\label{item:data-fh-exact}
      If $(\bar u,\bar x)$ is informative, then $(A,B)$ is
      approximately controllable if and only if no $\lambda,\kappa\in\CC$ and
      $\eta\in\mathcal D(A^\star)_{\CC}\setminus\{0\}$ satisfy
      \eqref{eq:data-fh}. Whenever such a triple exists,
      $\kappa=\angl{\eta}{x_0}_X\ne0$.
  \end{enumerate}
\end{thm}

The proof is given in Appendix~\ref{app:data-fattorini-hautus}.

The two parts trade the hypothesis on the record against the set that has to
be searched. Only the second part is an equivalence, but its search
set is the domain $\mathcal D(A^\star)$, which is not known from data. The first part enlarges the search to the ambient space
$X\supset\mathcal D(A^\star)$, which can only add candidates; the resulting
condition is more demanding, yet it is decided by the record alone and
requires no informativity.

The state projection $\angl{\eta}{\bar x(t)}_X$ can also be recovered from
input--output data, provided that $(A,C)$ is exactly observable: exact
observability makes every continuous state functional representable by a
bounded functional of a length-$T$ input--output window, whose input part
compensates for the forced response. Pairing the shifted windows of the record
with such a functional therefore takes the place of the state projection, and
does so at every admissible shift, giving the following counterparts of
Theorem~\ref{thm:data-fattorini-hautus}, which involve only the measured
pair $(\bar u,\bar y)$.

\begin{thm}[Input--output tests for approximate controllability]
  \label{thm:io-window-controllability}
  Suppose that the hypotheses of Proposition~\ref{prop:fattorini-hautus} hold,
  that Assumption~\ref{hyp:H2} holds, and that $(A,C)$ is exactly observable
  in time $T$. Let $T'>0$ and consider the exponential identity
  \begin{equation}\label{eq:io-window-fh}
    \int_0^T\bigl(\angl{v(t)}{\bar u(s+t)}_U
      +\angl{g(t)}{\bar y(s+t)}_Y\bigr)\,\d t
    =\kappa\,e^{\lambda s},
  \end{equation}
  the pairings being understood on the complexified spaces.
  \begin{enumerate}[label=\textup{(\roman*)},leftmargin=*]
    \item\label{item:io-fh-certificate}
      Let $(\bar u,\bar x,\bar y)$ be the record on $[0,T+T']$ generated by
      $(x_0,\bar u)\in X\times L^2(0,T+T';U)$, and suppose that
      $(\bar u,\bar x)$ is informative on $(0,T+T')$. If no
      $\lambda\in\CC$, $\kappa\in\CC\setminus\{0\}$,
      $v\in L^2(0,T;U_{\CC})$, and $g\in L^2(0,T;Y_{\CC})\setminus\{0\}$
      satisfy \eqref{eq:io-window-fh} for every $s\in(0,T')$, then
      $(A,B)$ is approximately controllable.
    \item\label{item:io-fh-exact}
      Let $(\bar u,\bar x,\bar y)$ be the record on $[0,3T+T']$ generated by
      $(x_0,\bar u)$, and suppose that $(\bar u,\bar x)$ is window informative
      at horizon $3T$. Then $(A,B)$ is approximately controllable if
      and only if no $\lambda\in\CC$, $\kappa\in\CC\setminus\{0\}$,
      $v\in L^2(0,T;U_{\CC})$, and $g\in L^2(0,T;Y_{\CC})$ satisfy
      \eqref{eq:io-window-fh} for every $s\in(0,2T+T')$. Whenever such a
      quadruple exists, $g\ne0$.
  \end{enumerate}
\end{thm}

The proof is given in Appendix~\ref{app:io-window-controllability}. For each
finite-dimensional discretization, approximate and exact observability
coincide, although the observability constant may deteriorate under refinement.

\section{Data-driven LQR}\label{sec:lqr}

We now turn to the finite-horizon linear--quadratic regulator (LQR). In the classical setting for infinite-dimensional systems, the optimal control is determined by an operator Riccati equation that requires full knowledge of the model \parencite{pritchardLinearQuadraticControl1987}. Recently, however, \textcite{schmitzContinuoustimeFundamentalLemma2024} demonstrated that for finite-dimensional systems, exactly the same optimal control can be extracted directly from a single informative data record. Using the fundamental lemma established in Section~\ref{sec:fundamental-lemma}, we now extend this data-driven principle to infinite dimensions, in an input--state--output and an input--output form, and show that both return the model-based optimal control.

\subsection{Model-based LQR}\label{subsec:lqr-riccati}

Fix a horizon $T>0$. Let $G\in\mathcal L(V,V')$ be self-adjoint and nonnegative, and let $R\in\mathcal L(U)$ be self-adjoint and coercive, say $\angl{u}{Ru}_U\ge\eps_R\|u\|_U^2$ for some $\eps_R>0$. For an initial state $x_0\in W$, the classical LQR problem seeks to minimize the cost
\begin{equation}\label{eq:lqr-cost}
  J(u;x_0) := \angl{x(T)}{Gx(T)}_{V,V'} + \int_0^T \bigl(\|Cx(t)\|_Y^2+\angl{u(t)}{Ru(t)}_U\bigr)\,\d t ,
\end{equation}
over $u\in L^2(0,T;U)$ subject to the system dynamics \eqref{eq:hilbert-control-formal}.

The Riccati theory of \textcite{pritchardLinearQuadraticControl1987} uses
the compatibility Assumption~\ref{hyp:compat} together with the admissibility
Assumptions~\ref{hyp:H1} and~\ref{hyp:H2}.

\begin{thm}[{\parencite[Thm.~2.7]{pritchardLinearQuadraticControl1987}}]\label{thm:lqr-riccati}
  Suppose Assumptions~\ref{hyp:H1}, \ref{hyp:H2}, and~\ref{hyp:compat} hold. Then the integral operator Riccati equation associated with \eqref{eq:lqr-cost}, with terminal condition $P(T)=G$, has a unique strongly continuous, self-adjoint, nonnegative solution $P:I\to\mathcal L(V,V')$. For every $x_0\in W$, the problem \eqref{eq:lqr-cost} has a unique optimal control and corresponding minimal cost given by
  \begin{equation}\label{eq:lqr-feedback}
    u_F(t)=-R^{-1}B^\star P(t)x(t), \qquad J(u_F;x_0)=\angl{x_0}{P(0)x_0}_{V,V'} .
  \end{equation}
\end{thm}

\subsection{The data-driven regulator}\label{subsec:lqr-data}

The model-based formulation relies on knowing $A$, $B$, and $C$. However, the cost function \eqref{eq:lqr-cost} can be rewritten along a trajectory using only the measured signals $(u, x, y)$:
\begin{equation}\label{eq:lqr-cost-traj}
  \mathcal J(u,x,y) := \angl{x(T)}{Gx(T)}_{V,V'} + \int_0^T\bigl(\|y(t)\|_Y^2+\angl{u(t)}{Ru(t)}_U\bigr)\,\d t .
\end{equation}
The system model enters only through the constraint that $(u,x,y)$ be a trajectory. Using the fundamental lemma, we can replace this dynamic constraint with a data-driven one derived entirely from a single record.

Given a pre-recorded input--state--output trajectory $(\bar u,\bar x,\bar y)$, let $\mathcal G:=\mathcal G(\bar u,\bar x,\dot{\bar x},\bar y)$ be its Gramian on $U\times W\times V_{-1}\times Y$. We define the set of data-compatible trajectories for an initial condition $x_0 \in W$ as
\begin{equation}\label{eq:lqr-admissible}
  \mathcal T_{\mathcal G}(x_0) := \left\{ (u,x,y) \;\middle|\;
    \begin{aligned}
      &u\in L^2(I;U),\ y\in L^2(I;Y),\ x\in C(I;W)\cap H^1(I;V_{-1}), \\
      &x(0)=x_0, \text{ and } (u,x,\dot x,y)(t)\in(\ker\mathcal G)^\perp \text{ for a.e.\ } t\in I
    \end{aligned}
  \right\}.
\end{equation}
The \emph{data-driven regulator} minimizes the cost over this graph constraint,
\begin{equation}\label{eq:lqr-data-problem}
  \min_{(u,x,y)\in\mathcal T_{\mathcal G}(x_0)}\ \mathcal J(u,x,y) .
\end{equation}

\begin{thm}[Data-driven LQR]\label{thm:lqr-data}
  Let Assumption~\ref{hyp:H1} hold. Let $(\bar x,\bar y)$ be the mild solution generated by $(\bar x_0,\bar u)$ with $\bar x_0\in W$, and suppose that $(\bar u,\bar x)$ is informative in $U\times W$. Then, for any $x_0 \in W$, the data-driven problem \eqref{eq:lqr-data-problem} has a unique minimizer $(u^\star,x^\star,y^\star)$, where $u^\star$ coincides with the unique minimizer of the model-based problem \eqref{eq:lqr-cost}, and $(x^\star,y^\star)$ is its corresponding mild solution.

  If Assumptions~\ref{hyp:H2} and~\ref{hyp:compat} also hold, the data-driven solution recovers the Riccati-based optimal control and cost exactly:
  \[
    u^\star(t)=-R^{-1}B^\star P(t)x^\star(t), \qquad \min_{(u,x,y)\in\mathcal T_{\mathcal G}(x_0)}\mathcal J(u,x,y) = \angl{x_0}{P(0)x_0}_{V,V'} .
  \]
\end{thm}

The proof is given in Appendix~\ref{app:lqr-data}.

\subsection{The input--output regulator}\label{subsec:lqr-data-io}

Problem \eqref{eq:lqr-data-problem} reads the state twice: in the record,
through the Gramian $\mathcal G$, and in the initial condition $x(0)=x_0$. When
only input--output data is available, neither is at hand. Following
Section~\ref{subsec:io-behavior}, we let a measured window take the place of the
state: we fix a conditioning horizon $T_0>0$, prescribe an input--output window
$(u_{\rm ini},y_{\rm ini})\in L^2(0,T_0;U\times Y)$, and regulate over the next
$T$ units of time. The optimization then runs over input--output pairs of length
$T_0+T$, the record being measured on $[0,T_0+T+T']$. Since the terminal state
is not measured, we set $G=0$; the \emph{input--output regulator} associated
with a pre-recorded trajectory $(\bar u,\bar y)$ is
\begin{equation}\label{eq:lqr-data-io-problem}
  \begin{aligned}
    \min_{(u,y)}\quad
      &\int_0^T\bigl(\|y(T_0+t)\|_Y^2
        +\angl{u(T_0+t)}{Ru(T_0+t)}_U\bigr)\,\d t\\
    \text{subject to}\quad
      &(u,y)\in\overline{\im\mathcal Y_{T_0+T,T'}^{\bar u,\bar y}},
      \qquad
      (u,y)\vert_{(0,T_0)}=(u_{\rm ini},y_{\rm ini}) .
  \end{aligned}
\end{equation}
The two constraints take over the two roles of the state. The first is the
data-driven form of the dynamics, licensed by
Theorem~\ref{thm:windowed-io-fundamental-lemma}; the second is the data-driven
form of the initial condition, since by Lemma~\ref{lem:io-parametrization}
the conditioning window determines the state at time $T_0$ as soon as $(A,C)$ is
observable in time $T_0$. Neither involves the state itself.

\begin{thm}[Data-driven LQR, input--output]\label{thm:lqr-data-io}
  Suppose Assumptions~\ref{hyp:H1}, \ref{hyp:H2}, and~\ref{hyp:compat} hold,
  that $G=0$, and that $(A,C)$ is exactly observable in time $T_0$. Let
  $(\bar u,\bar x,\bar y)$ be the record on $[0,T_0+T+T']$ generated by
  $(\bar x_0,\bar u)\in X\times L^2(0,T_0+T+T';U)$, and suppose that
  $(\bar u,\bar x)$ is window informative at horizon $T_0+T$. Let
  $(u_{\rm ini},y_{\rm ini})$ be the input--output window generated by
  $(x_{\rm ini},u_{\rm ini})\in W\times L^2(0,T_0;U)$, and set
  $x_0:=x(T_0;x_{\rm ini},u_{\rm ini})\in W$. Then the input--output problem
  \eqref{eq:lqr-data-io-problem} has a unique minimizer $(u^\star,y^\star)$,
  whose regulated input $u^\star(T_0+\cdot)$ is the unique minimizer of the
  model-based problem \eqref{eq:lqr-cost} for the initial state $x_0$. In
  particular, the two problems have the same optimal value.
\end{thm}

The proof is given in Appendix~\ref{app:lqr-data-io}. 

\section{Numerical experiments}\label{sec:numerics}

We illustrate the constructive results of the paper on the three examples from
the introduction, which represent the three families of
\textcite[§4]{pritchardLinearQuadraticControl1987}. 
We use them to test the state and input--output controllability criteria,
compare graph-based and windowed realizations of the LQR, and examine the
harmonic input of Lemma~\ref{lem:harmonic-pe}. Throughout, \emph{i-s}, \emph{i-s-o}, and
\emph{i-o} label a method by the data it reads: input--state, input--state--output,
and input--output. All values, figures, and tables are generated by
the accompanying Python code.\footnote{Available at
\url{https://github.com/DCN-FAU-AvH/data-driven-control}.}


The PDEs and the delay equation are discretized by continuous piecewise-linear
finite elements on a uniform mesh; time integration uses
Crank--Nicolson on $t_k=k\Delta t$. Inner products are evaluated in the
corresponding discrete $X$- or $W$-metric, and throughout $m=\dim U=1$. The
meshes carry between $4$ and $24$ elements, the horizons are of order $10$,
and $\Delta t$ ranges from $10^{-3}$ to $2\cdot10^{-2}$; these settings, the
probing signals, and the tolerances are fixed in the accompanying code.

The outputs are those of the introduction: point observation at
$\xi_0=0.6$ for heat, a smooth integral observation for wave, and the delayed
first coordinate for the delay system. None of these continuum pairs is
exactly observable, and discretization need not restore observability: the
symmetric heat and wave sensors miss the same spatial modes on the meshes
used here. The input--output experiments therefore test only the observable,
numerically resolved part of the behavior, and do not verify all hypotheses
of Theorem~\ref{thm:io-window-controllability} on the full state space.

We use the following convention for the spectra. Let $\mu_1\ge\mu_2\ge\cdots\ge\mu_d\ge0$
denote either the eigenvalues of a metric-weighted Gramian or the squared
singular values of a metric-weighted moment matrix, as specified in each
experiment. At the fixed relative tolerance $\varepsilon$, define
\begin{equation}\label{eq:num-resolved}
  r_\varepsilon:=\#\{j:\mu_j>\varepsilon\mu_1\}.
\end{equation}
We call $r_\varepsilon$ the number of \emph{resolved directions} and use
$\varepsilon=10^{-10}$ throughout, the two LQR implementations applying this
tolerance at different powers.
Full numerical rank is the discrete counterpart of Gramian injectivity in
Lemma~\ref{lem:informative-gramian}. In the infinite-dimensional limit, the
Gramian is compact and its eigenvalues approach zero, so $r_\varepsilon<d$
means that the record supports conclusions only on the resolved subspace; it
does not classify the remaining directions.

\subsection{Data-driven controllability test}
\label{subsec:num-controllability}

We set $\nu=0.02$ and $c=1/2$, so that the leading heat and wave modes lie in
the band of the probing multisine. For the sine test functions
$\varphi_i(t)=\sin(i\pi t/T)\in H_0^1(I)$, we form the data moments
\[
  X_0^i=\int_I\varphi_i \bar x\,\d t,\qquad
  X_1^i=-\int_I\dot\varphi_i \bar x\,\d t,\qquad
  U_0^i=\int_I\varphi_i \bar u\,\d t .
\]
Let $U_0$, $X_0$, and $X_1$ denote the matrices whose columns are these
moments. They satisfy $X_1=A_hX_0+B_hU_0$ up to time-integration and
quadrature error, the relative residual being below $2\cdot10^{-5}$ in all
six runs. For the resolved count in
\eqref{eq:num-resolved}, the $\mu_j$ are the squared singular values of
$\col(U_0,X_0)$ in the metric of $U\times H$, with $H=X$ for
heat and $H=W$ for wave and delay. Here $d=d_h:=m+n$.

By $X_1=A_hX_0+B_hU_0$, a covector $\eta$ with $A_h^\top\eta=\lambda\eta$
and $B_h^\top\eta=0$ satisfies $\eta^\top X_1=\lambda\,\eta^\top X_0$, and the
converse holds once $\col(U_0,X_0)$ has full row rank. Following
Theorem~\ref{thm:data-fattorini-hautus}\,\ref{item:data-fh-exact} we
therefore take the left eigenvectors of the pencil $(X_1,X_0)$ as candidates
and fit $\eta^\top\bar x(t)=\kappa\e^{\lambda t}$ to each.
The \emph{found} column of Table~\ref{tab:num-controllability} reports the
retained candidates in the \emph{i-s} rows.

Each system is tested in a controllable and an uncontrollable configuration.
The one-sided heat and wave controls reach every mode. Their symmetric
comparisons apply the same scalar control at both endpoints and cannot reach
the modes antisymmetric about $\xi=1/2$; the heat comparison imposes Dirichlet
data at both ends, which is why its state dimension differs from the one-sided
case. The delay model uses $M=0$ and
$L\phi=A_0\phi(0)+A_1\phi(-h)$, with
$A_0=\left(\begin{smallmatrix}-1&1/2\\ a&-1/2\end{smallmatrix}\right)$ and
$A_1=\left(\begin{smallmatrix}-3/10&1/5\\ b&-2/5\end{smallmatrix}\right)$, with
$(a,b)=(2/5,1/4)$ coupled and $(0,0)$ decoupled; in the latter the second
coordinate receives neither the input nor any coupling from the first.

\paragraph{The same test from input--output data.}
The \emph{i-o} rows test the exponential identity of
Theorem~\ref{thm:io-window-controllability}\,\ref{item:io-fh-exact}
on the same six configurations, over shifted windows of length $T$. If $n_w$ is
the number of samples in one such window, we use the nominal dimension
$d=b_h:=mn_w+n$ as a reference; the sampled behavior is smaller when the
discrete pair is unobservable. Since
\eqref{eq:io-window-fh} requires $\e^{\lambda s}$ to lie in the span of the
window pairings, we evaluate
\begin{equation*}
  \varrho(\lambda):=\frac{\operatorname{dist}_{L^2(0,2T+T')}
    \bigl(\e^{\lambda s},\ \mathcal P\bigr)}
    {\|\e^{\lambda s}\|_{L^2(0,2T+T')}},
  \quad
  \mathcal P:=\Bigl\{s\mapsto\int_0^T\bigl(\angl{v(t)}{\bar u(s+t)}_U
    +\angl{g(t)}{\bar y(s+t)}_Y\bigr)\,\d t\Bigr\}
\end{equation*}
where $v$ and $g$ range over $L^2(0,T;U_{\CC})$ and $L^2(0,T;Y_{\CC})$, the
candidates $\lambda$ being the eigenvalues of the shift pencil formed on the
resolved window directions, refined locally against $\varrho$. We retain a
candidate when $\varrho$ falls below $3\cdot10^{-3}$.

The two tests use different excitations. A multisine probes the resolved
directions for the state test, but its finitely many lines do not provide a
spanning input-window library. The \emph{i-o} records therefore use a
pseudorandom binary signal (PRBS).

\begin{table}[!htbp]
  \centering
  \caption{Controllability tests from input--state (i-s) and input--output
  (i-o) records. \emph{Found} counts the detected obstructions. The marks give
  the model-based status.}
  \label{tab:num-controllability}
  {\small\setlength{\tabcolsep}{4pt}\begin{tabular}{l@{\hspace{4pt}}clrll}
\toprule
\multicolumn{2}{l}{Configuration} & data & found & recovered $\lambda$ &
exact $\lambda$ \\
\midrule
heat, one-sided & \cmark & i-s & 0 & --- & --- \\
 &  & i-o & 0 & --- & --- \\
\addlinespace
heat, symmetric & \xmark & i-s & 1 & $-0.783$ & $-0.790$ \\
 &  & i-o & 1 & $-0.783$ & $-0.790$ \\
\addlinespace
wave, one-sided & \cmark & i-s & 0 & --- & --- \\
 &  & i-o & 0 & --- & --- \\
\addlinespace
wave, symmetric & \xmark & i-s & 2 & $0\pm3.129\mathrm{i}$ & $0\pm3.142\mathrm{i}$ \\
 &  & i-o & 2 & $0\pm3.128\mathrm{i}$ & $0\pm3.142\mathrm{i}$ \\
\addlinespace
delay, coupled & \cmark & i-s & 0 & --- & --- \\
 &  & i-o & 0 & --- & --- \\
\addlinespace
delay, decoupled & \xmark & i-s & 13 & $-1.102\pm1.051\mathrm{i}$ & $-1.106\pm1.046\mathrm{i}$ \\
 &  & i-o & 2 & $-1.101\pm1.048\mathrm{i}$ & $-1.106\pm1.046\mathrm{i}$ \\
\bottomrule
\end{tabular}
}
\end{table}

Both tests recover the expected obstruction in every uncontrollable
configuration and none in the resolved subspaces of the controllable ones.
This absence of detection is not a certificate beyond the resolved subspace.
For the representative exponents in Table~\ref{tab:num-controllability},
the two tests agree within $0.2\%$ in relative complex modulus, and both
agree with the continuum modes within $1\%$. The window test detects fewer
roots of the decoupled delay system because a finite output window resolves
fewer modes than direct state measurements.



\subsection{Data-driven LQR}
\label{subsec:num-lqr}

We compare two realizations of the data-driven regulator. The graph method,
labeled \emph{i-s-o}, uses input, state, and output data together with the
synthesis characterization~\ref{item:wfl-synthesis} of
Theorem~\ref{thm:willems-gramian}. The window method, labeled \emph{i-o}, uses
only input--output data and minimizes the cost over superpositions of measured
windows, as licensed by Theorem~\ref{thm:windowed-io-fundamental-lemma}.
Theorem~\ref{thm:lqr-data-io} establishes its optimality under exact
observability in time $T_0$, which fails for the continuum examples. The
coarse discretizations used here are observable, but the implementation
truncates the window library and penalizes rather than enforces the
conditioning match, so the guarantee does not carry over verbatim: the
penalty can lower the reported cost, whereas the truncation can raise it.

We use $\nu=1$ for the heat equation because at $\nu=0.02$ its optimal input is
nearly zero. We take $R=\tfrac1{20}\operatorname{id}$ and $G=0$; since the terminal cost vanishes, both methods
minimize the same functional of the measured signals. A preceding conditioning
window of length $T_0$ determines the initial state $x_0$ of each
discretized plant.

For $q$ test functions, let $U_0$, $X_0$, $X_1$, and $Y_0$ denote the weak
moments of the measured input, state, state derivative, and output. The graph
method forms
\[
  Z_q=\col(U_0,X_0),
  \qquad
  D_q=\col(U_0,X_0,X_1,Y_0).
\]
With moments taken consistently with Crank--Nicolson, $D_q=\Gamma_hZ_q$ holds
up to solver accuracy for the graph operator \eqref{eq:graph-operator} of the
discretized system. Thus, when $Z_q$ has full row rank $m+n$,
$\im D_q=\im\Gamma_h$. A metric-weighted SVD gives an annihilator $K$ with
$\ker K=\im D_q$. On the grid $t_k=k\Delta t$, the resulting regulator is
the quadratic program
\begin{equation}\label{eq:lqr-graph-discrete}
  \begin{aligned}
    \min_{x_k,u_k,y_k}\quad&
    \angl{x_N}{Gx_N}
    +\Delta t\sum_{k=0}^{N-1}
      \bigl(\|y_k\|_Y^2+\angl{u_k}{Ru_k}_U\bigr),\\
    \text{subject to}\quad&x_0=x_{0,h},\quad K
    \begin{bmatrix}
      u_k\\
      (x_k+x_{k+1})/2\\
      (x_{k+1}-x_k)/\Delta t\\
      y_k
    \end{bmatrix}=0,
    \qquad k=0,\ldots,N-1 .
  \end{aligned}
\end{equation}
At full row rank, this constraint is the Crank--Nicolson discretization of
the dynamics expressed through data rather than $A_h$, $B_h$, and $C_h$.

The window method uses $p$ shifted input--output windows of length $T_0+T$.
With the time origin placed at the end of the conditioning window, exact
matching of their restriction to $[-T_0,0]$ would fix the discrete initial
state; the implementation instead minimizes the future trajectory cost plus
$\rho^{-1}$ times the squared $L^2$ conditioning error, with $\rho=10^{-8}$.
The reported $J_{\rm dd}$ excludes that penalty.

The graph has dimension $d_h=m+n$, and the nominal sampled-window behavior
dimension is $b_h=mn_w+n$, the windows now being of length $T_0+T$ so that
$n_w$ counts the samples of that longer window. We use the spanning trial sizes
\[
  q=2d_h,\qquad p=2b_h.
\]
The Riccati solution provides the reference cost $J_\star$. We report
\[
  e_J=\frac{|J_{\rm dd}-J_\star|}{J_\star},
  \qquad
  e_J^{\rm plant}=\frac{|J(u_{\rm dd};x_0)-J_\star|}{J_\star} .
\]

At $q=2d_h$ all three graph matrices have full resolved rank; $e_J$ is
negligible and $e_J^{\rm plant}$ sits at the level of the time discretization.
At $p=2b_h$ the window libraries resolve all but a few nominal directions,
the shortfall reflecting the conditioning of the sampled library rather than
the closure in Theorem~\ref{thm:windowed-io-fundamental-lemma}. In
Table~\ref{tab:num-lqr}, the window values of $e_J$ are two to five orders of
magnitude above the graph values; the two implementations also truncate at
different thresholds. That $e_J$ and $e_J^{\rm plant}$ agree closely indicates
that the reconstructed costs are consistent with the plant.

The heat refinement in Figure~\ref{fig:num-lqr} gives observed convergence
orders $2.00$ and $1.00$ for the graph and window methods, respectively. The latter inherits one measured input sample
at the junction $t=0$, producing an $O(\Delta t)$ perturbation; otherwise both
methods closely reproduce the Riccati trajectory in the plotted example.
The wave library has full resolved rank, so its residual error comes from
conditioning, the penalized past match, and time discretization rather than
from unresolved directions.

\begin{figure}[!htbp]
  \centering
  \includegraphics[width=\textwidth]{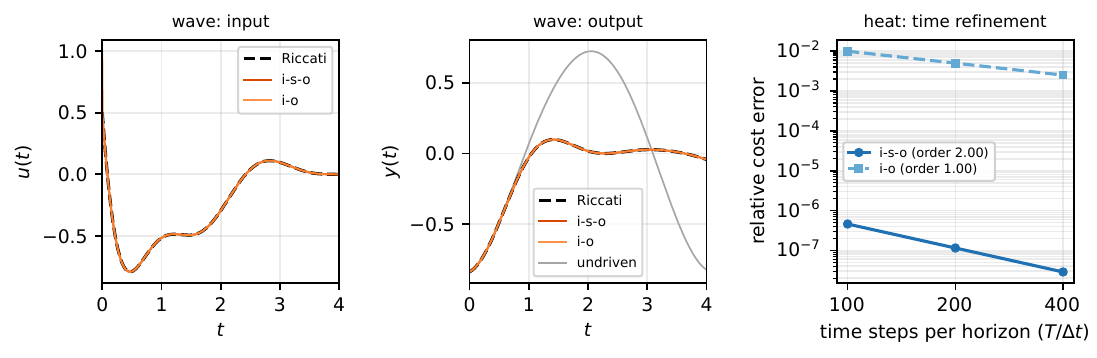}
  \caption{Data-driven finite-horizon LQR. Left and center: wave input and
  output compared with the Riccati solution; the gray curve is the uncontrolled
  output. Right: heat cost error versus the number of time steps.}
  \label{fig:num-lqr}
\end{figure}

\begin{table}[!htbp]
  \centering
  \caption{LQR results at $q=2d_h$ and $p=2b_h$; $e_J^{\rm plant}$ evaluates the
  recovered input directly on the plant.}
  \label{tab:num-lqr}
  \begin{tabular}{llrrr}
\toprule
Example & data & $J_{\rm dd}$ & $e_J$ & $e_J^{\rm plant}$ \\
\midrule
heat & i-s-o & \num{0.291} & \num{2.872e-08} & \num{2.150e-05} \\
heat & i-o & \num{0.292} & \num{2.503e-03} & \num{2.502e-03} \\
\addlinespace
wave & i-s-o & \num{0.323} & \num{1.443e-06} & \num{7.068e-05} \\
wave & i-o & \num{0.323} & \num{2.861e-04} & \num{2.912e-04} \\
\addlinespace
delay & i-s-o & \num{7.624e-03} & \num{3.796e-07} & \num{5.102e-05} \\
delay & i-o & \num{7.898e-03} & \num{0.0359} & \num{0.0361} \\
\bottomrule
\end{tabular}

\end{table}

\subsection{Excitation and conditioning}
\label{subsec:num-conditioning}

Finally, we compare a ten-frequency truncation of the harmonic input in
Lemma~\ref{lem:harmonic-pe} against a well-separated multisine and a PRBS; see
Figure~\ref{fig:num-conditioning}. For this experiment, the $\mu_j$ in
\eqref{eq:num-resolved} are the eigenvalues of the input--state Gramian
$\mathcal G(u,x)$ on $U\times X$ for heat and $U\times W$ for wave.

The heat equation is of the analytic type covered by
Example~\ref{ex:boundary-heat-informative}; the wave group is nonanalytic and
is included as a comparison. The accumulation of frequencies used in the
proof of harmonic informativity has a numerical cost: on a finite time
interval, nearby sinusoids are nearly linearly dependent. On the heat equation
the harmonic and multisine inputs therefore resolve the same $9$ of $26$
directions, and the PRBS two more; of the two sinusoidal signals, the
multisine keeps its weakest resolved eigenvalue nearly five times above the
harmonic input's. For the non-smoothing wave equation the distinction
disappears: all three inputs resolve every direction, and their spectra decay
by only a few orders of magnitude, against the ten orders that separate the
heat spectra from the tolerance. This comparison concerns finite-precision conditioning;
Theorem~\ref{thm:analytic-universal-sufficiency} concerns density for the
untruncated harmonic signal.

\begin{figure}[!htbp]
  \centering
  \includegraphics[width=\textwidth]{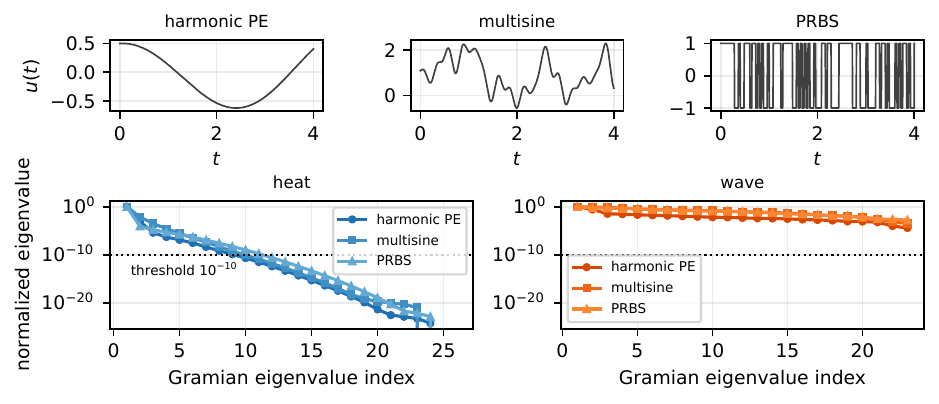}
  \caption{Inputs and Gramian conditioning. Top: $u(t)$ over the first four
  units of the observation interval. Bottom: normalized Gramian eigenvalues
  $\mu_j/\mu_1$, ordered by index $j$, for heat (left) and wave (right). Values
  above the dotted line count toward $r_\varepsilon$ in
  \eqref{eq:num-resolved}.}
  \label{fig:num-conditioning}
\end{figure}

\section{Conclusions and open problems}
\label{sec:conclusion}

We have extended data informativity, persistency of excitation, and Willems'
fundamental lemma to semigroup dynamics with unbounded control and observation
operators. The central finding is that persistency of excitation, while still
necessary for a universally informative input, is---unlike in finite
dimensions---no longer sufficient; harmonic persistency of excitation restores
sufficiency for analytic semigroups. We also derived data-driven controllability
tests and showed how a single informative record can be used to formulate a
data-driven LQR problem.

On the numerical side, we tested the proposed controllability test and regulator, and assessed the conditioning of the harmonic excitation, on representative examples involving linear PDEs and delay differential equations.

Several questions remain open. The convergence and stability properties of the numerical schemes require a more detailed analysis. The input--output theory developed here should be extended to allow weaker observability assumptions and informativity conditions that can be verified from input--output measurements alone. A central challenge is to identify sufficient conditions for informativity beyond analytic semigroups. More broadly, the theory of data informativity \parencite{vanwaardeDataInformativityNew2020, vanwaardeInformativityApproachDataDriven2023} should be further extended to the infinite-dimensional setting, including properties such as stabilizability and dissipativity.

\section*{Acknowledgments}

The author is deeply grateful to Enrique Zuazua for his insightful feedback,
 encouragement, and continuous support throughout the development of
this work.

\subsection*{Funding}

This work has been supported by the Alexander von Humboldt-Professorship
program of the Alexander von Humboldt Foundation, funded by the German Federal
Ministry of Education and Research (BMBF), at the Chair for Dynamics, Control,
Machine Learning, and Numerics, Friedrich-Alexander-Universit\"at
Erlangen-N\"urnberg.

\appendix
\section{Proofs}
\label{app:proofs}

\subsection{Proof of Lemma~\ref{lem:finite-horizon-output} (finite-horizon input--output operators)}
\label{app:finite-horizon-output}
Translating an input on $(0,t)$ to the terminal part of $(0,T)$ in
Assumption~\ref{hyp:H1} gives
\[
  \|L_tu\|_W\le b\|u\|_{L^2(0,t;U)},\qquad t\le T.
\]
Repeated interval splitting and the semigroup on $W$ therefore yield, for
each finite $\tau$, a constant $b_\tau$ such that
$\sup_{t\le\tau}\|L_tu\|_W\le b_\tau\|u\|_{L^2(0,\tau;U)}$. Extending
$u$ by zero outside $(0,\tau)$, write $L_tu=L_\tau u_t$, where
$u_t(r)=u(t-\tau+r)$ for $0<r<\tau$. Continuity of translations in
$L^2$ then gives $L_\cdot u\in C([0,\tau];W)$. Boundedness of
$\mathcal F_\tau$ follows from $C\in\mathcal L(W,Y)$.  The same splitting,
now using Assumption~\ref{hyp:H2} and the semigroup bounds, proves
boundedness of $\mathcal O_\tau$.  For $x_0\in W$, the output formula and
shift identity follow by splitting the mild solution at $s$.  Density of
$W$ in $X$ and boundedness of the operators extend both identities to
$x_0\in X$.
\qed

\subsection{Compatible control operator on the state space}
\label{app:control-realization}

Let $A_V$ and $A_{X,-1}$ be the generators on $V$ and $X_{-1}$, respectively.
For sufficiently large real $r$, set $Q_r:=(r\operatorname{id}-A_V)^{-1}B$.
The finite-time Laplace identity gives
\[
  (\operatorname{id}-e^{-rT}S(T))Q_rv=L_Tu_r,
  \qquad u_r(t):=e^{-r(T-t)}v.
\]
The operator on the left is invertible on both $V$ and $W$, with consistent
inverses. Assumption~\ref{hyp:H1} therefore gives $Q_r\in\mathcal L(U,W)$ and
\begin{equation}\label{eq:control-resolvent-decay}
  \|Q_r\|_{\mathcal L(U,W)}
  \le b\|(\operatorname{id}-e^{-rT}S(T)|_W)^{-1}\|
       \sqrt{\frac{1-e^{-2rT}}{2r}}
  =O(r^{-1/2}).
\end{equation}
For one such $a$, define
\[
  B_X:=(a\operatorname{id}-A_{X,-1})Q_a\in\mathcal L(U,X_{-1}).
\]
The resolvent identity, using consistency of the resolvents on $X$ and $V$
for large real parameters, shows that this definition is independent of $a$
and that $(r\operatorname{id}-A_{X,-1})^{-1}B_X=Q_r$ for all sufficiently
large $r$. Equality of Laplace transforms in $X_{-1}$ then identifies its
control convolution with $L_t$, first for compactly supported inputs and
hence on every finite interval. Finally, for a constant input $v$, the vector
$t^{-1}L_tv$ converges to $Bv$ in $V$ and to $B_Xv$ in $X_{-1}$ as $t\downarrow0$.
Pairing with any $\eta\in V'\cap\mathcal D(A^\star)$ proves that the two
adjoints agree on this intersection. The adjoint of $B_X$ is bounded on
$\mathcal D(A^\star)$ equipped with its graph norm, since this is the pivot
dual of $X_{-1}$ \parencite[Sec.~2.10]{tucsnakObservationControlOperator2009}.

\subsection{Proof of Lemma~\ref{lem:informative-gramian} (Gramian characterization of informative data)}
\label{app:informative-gramian}
The integrand in \eqref{eq:hilbert-gramian} has trace norm $\|z(t)\|_H^2$,
which is integrable on $I$, so the integral converges as a Bochner integral
in the trace class. As an integral of nonnegative, self-adjoint rank-one
operators, $\mathcal G(z)$ is itself self-adjoint, nonnegative, and
trace-class. For $\zeta\in H$,
\[
  \angl{\mathcal G(z)\zeta}{\zeta}_H
    =\int_0^T|\angl{z(t)}{\zeta}_H|^2\,\d t ,
\]
so $\mathcal G(z)\zeta=0$ if and only if $\angl{z(t)}{\zeta}_H=0$ for a.e.\
$t\in I$, which is \eqref{eq:gramian-kernel}. Applying this with $U\times H$
in place of $H$ and $z=(\bar u,\bar x)$, injectivity of $\mathcal G(\bar u,\bar x)$ is
therefore equivalent to $(\cspan(\bar u,\bar x))^{\perp}=\{0\}$, that is, to
\eqref{eq:informative-span}. In the finite-dimensional case the closure is
superfluous and the algebraic span suffices.
\qed

\subsection{Proof of Lemma~\ref{lem:window-implies-informative} (window informativity implies informativity)}
\label{app:window-implies-informative}
Let $(\alpha,\beta)\in U\times X$ annihilate the record, so that
$\angl{\alpha}{\bar u(t)}_U+\angl{\beta}{\bar x(t)}_X=0$ for a.e.\
$t\in(0,T+T')$. By translation invariance this holds, for \emph{every}
$s\in(0,T')$, for a.e.\ $t\in(0,T)$. Fix $\phi\in L^2(0,T)$, multiply by
$\phi(t)$ and integrate over $(0,T)$. The input term becomes
$\angl{\phi\alpha}{\bar u(s+\cdot)}_{L^2(0,T;U)}$, and the shift identity
$\bar x(s+t)=S(t)\bar x(s)+L_t\bar u(s+\cdot)$ of
Lemma~\ref{lem:finite-horizon-output} splits the state term into
$\angl{\xi_\phi}{\bar x(s)}_X+\Lambda_\phi(\bar u(s+\cdot))$, where
$\xi_\phi:=\int_0^T\phi(t)S(t)^\star\beta\,\d t$ (a Bochner integral, since
$t\mapsto S(t)^\star\beta$ is strongly continuous and bounded) and
$\Lambda_\phi(v):=\int_0^T\phi(t)\angl{\beta}{L_tv}_X\,\d t$. By
Assumption~\ref{hyp:H1} and $W\hookrightarrow X$ we have
$|\Lambda_\phi(v)|\le b_T'\|\beta\|_X\|\phi\|_{L^1(0,T)}\|v\|_{L^2(0,T;U)}$,
so $\Lambda_\phi$ is represented by some $g_\phi\in L^2(0,T;U)$. Hence
$(\phi\alpha+g_\phi,\xi_\phi)\in\ker(\mathcal Z_{T,T'}^{\bar u,\bar x})^\star$,
and window informativity forces $\phi\alpha+g_\phi=0$ and
$\xi_\phi=0$. The latter holds for every $\phi\in L^2(0,T)$, so the
continuous function $t\mapsto\angl{\beta}{S(t)^\star\beta}_X$ vanishes
on $(0,T)$. Letting $t\downarrow0$ gives $\|\beta\|_X^2=0$. Then
$\Lambda_\phi=0$, so $g_\phi=0$ and $\phi\alpha=0$; choosing $\phi\ne0$
gives $\alpha=0$.
\qed

\subsection{Verification of Example~\ref{ex:window-strictly-stronger} (informativity does not imply window informativity)}
\label{app:window-strictly-stronger}
Since $\|\bar u(t)\|_X^2=\sum_{n\ge1}4^{-n}\sin^2(nt)\le\tfrac13$, the input
\eqref{eq:window-strictly-stronger-input} is bounded and continuous. With
$A=0$ and $B=\operatorname{id}$ we have $S(t)=\operatorname{id}$,
$L_tu=\int_0^tu(r)\,\d r$ and $\|L_tu\|_X\le\sqrt t\,\|u\|_{L^2(0,t;U)}$, so
Assumption~\ref{hyp:H1} holds with $W=V=X$, and
$\RR_\tau(0)=\{\int_0^\tau u:u\in L^2\}=X$ for every $\tau>0$. From
$\bar x_0=0$,
\begin{equation}\label{eq:wss-state}
  \bar x(t)=\int_0^t\bar u(r)\,\d r
    =\sum_{n\ge1}\frac{2^{-n}}{n}\bigl(1-\cos(nt)\bigr)e_n .
\end{equation}

\emph{The record is informative in $U\times X$.} Fix $\tau>0$ and let
$(\alpha,\beta)\in U\times X$ annihilate the record on $(0,\tau)$; write
$\alpha_n:=\angl{\alpha}{e_n}_X$ and $\beta_n:=\angl{\beta}{e_n}_X$. By
\eqref{eq:window-strictly-stronger-input} and \eqref{eq:wss-state},
\[
  0=\sum_{n\ge1}2^{-n}
    \Bigl(\alpha_n\sin(nt)+\frac{\beta_n}{n}\bigl(1-\cos(nt)\bigr)\Bigr)
  \qquad\text{for a.e.\ }t\in(0,\tau) ,
\]
both series converging absolutely and uniformly because
$(\alpha_n),(\beta_n)\in\ell^2$. Writing the trigonometric functions as
exponentials and setting $q:=e^{it}$, the right-hand side becomes $g(q)$ for
the Laurent series
\[
  g(q):=c+\sum_{n\ge1}a_nq^n+\sum_{n\ge1}b_nq^{-n},
  \qquad
  c:=\sum_{n\ge1}\frac{2^{-n}\beta_n}{n},
\]
\[
  a_n:=2^{-n}\Bigl(\frac{\alpha_n}{2i}-\frac{\beta_n}{2n}\Bigr),
  \qquad
  b_n:=-2^{-n}\Bigl(\frac{\alpha_n}{2i}+\frac{\beta_n}{2n}\Bigr),
\]
which converges and is holomorphic on the annulus
$\{\tfrac12<|q|<2\}$. It vanishes on the arc
$\{e^{it}:t\in(0,\tau)\}$ up to a null set, hence on a subset of the annulus
with a limit point there; the identity theorem gives $g\equiv0$, and
uniqueness of Laurent coefficients gives $c=0$ and $a_n=b_n=0$ for every
$n$. Adding the last two yields $\beta_n=0$ and subtracting them yields
$\alpha_n=0$, so $(\alpha,\beta)=0$.

\emph{The record is window informative for no $T,T'>0$.} Fix
$T,T'>0$. Since $\dim L^2(0,T)=\infty$, there is
$\psi\in L^2(0,T)\setminus\{0\}$ with
$\int_0^T\psi(t)\sin t\,\d t=\int_0^T\psi(t)\cos t\,\d t=0$. Put
$g:=\psi e_1\in L^2(0,T;U)\setminus\{0\}$ and $\xi:=0$. Because
$\bar u(s+t)$ has $e_1$-component
$\tfrac12\bigl(\cos s\,\sin t+\sin s\,\cos t\bigr)$,
\[
  \int_0^T\angl{g(t)}{\bar u(s+t)}_U\,\d t
  =\frac{\cos s}{2}\int_0^T\!\!\psi(t)\sin t\,\d t
   +\frac{\sin s}{2}\int_0^T\!\!\psi(t)\cos t\,\d t
  =0
\]
for every $s$. Hence $(g,0)$ is a nonzero element of
$\ker(\mathcal Z_{T,T'}^{\bar u,\bar x})^\star$, so
$\overline{\im\mathcal Z_{T,T'}^{\bar u,\bar x}}\ne L^2(0,T;U)\times X$.
\qed

\subsection{Proof of Proposition~\ref{prop:hilbert-pe-necessary} (necessity of persistency of excitation)}
\label{app:pe-necessity}

In the discrete-time finite-dimensional setting, the corresponding necessity
is due to \textcite[Lem.~5]{shakouriNewPerspectiveWillems2025}; for
infinite-dimensional systems the argument is more involved.

\begin{proof}
  We argue by contraposition: assuming that $\bar u$ fails the stated
  excitation, we construct a pair $(A,B)$ in the admitted class and a solution
  $\bar x$ for which $(\bar u,\bar x)$ is not informative. The
  obstruction is finite-dimensional in both cases; the dimension of $X$ enters
  only through the modes that complete the construction, and it is only there
  that the case $N=\infty$ is felt.

  \emph{Step 1: the failing relation.} If $\bar u$ is not persistently
  exciting of the stated order, then there are an integer $r\ge0$ and
  $\zeta_0,\dots,\zeta_r\in U$ with $\zeta_r\neq0$ and a terminating relation
  \[
    \sum_{k=0}^{r}\angl{\zeta_k}{D^k\bar u}_U=0
    \quad\text{in }\DD'(I),
    \qquad r\le N
  \]
  (when $N=\infty$ this merely asserts that $r$ is finite). We encode it in the
  $U$-valued polynomial $\Psi$ and record its real zero set $Z$,
  \[
    \Psi(s):=\sum_{k=0}^{r}\zeta_k s^k,
    \qquad
    Z:=\{\lambda\in\R:\Psi(\lambda)=0\}.
  \]
  The scalar polynomial $s\mapsto\angl{\zeta_r}{\Psi(s)}_U$ has leading
  coefficient $\|\zeta_r\|_U^2\neq0$, hence does not vanish identically; since
  it vanishes at every point of $Z$, the set $Z$ is finite. The obstruction
  below lives on $r$ modes only; when $r=0$ every product and sum indexed by
  $1\le n\le r$ is empty, so that $p\equiv1$ and $\beta=0$.

  \emph{Step 2: the obstruction block.} Fix an orthonormal basis
  $(e_n)_{n=1}^{N}$ of $X$ and choose distinct reals
  $\lambda_1,\dots,\lambda_r\in(0,1)\setminus Z$, which is possible since $Z$ is
  finite and $r\le N$. Put
  \[
    p(s):=\prod_{n=1}^{r}(s-\lambda_n),
    \qquad
    \alpha:=\zeta_r,
    \qquad
    b_n:=\frac{\Psi(\lambda_n)}{p'(\lambda_n)}\neq0,
    \quad 1\le n\le r,
  \]
  where $b_n\ne0$ because $\lambda_n\notin Z$. The polynomial
  \[
    Q(z):=\alpha\,p(z)+\sum_{n=1}^{r}b_n\,\frac{p(z)}{z-\lambda_n}
  \]
  has degree at most $r$ and leading coefficient $\zeta_r$, the second sum
  contributing only in degrees below $r$, and $Q(\lambda_n)=b_n\,p'(\lambda_n)
  =\Psi(\lambda_n)$ for $1\le n\le r$. Hence $Q-\Psi$ has degree at most $r-1$
  and vanishes at the $r$ distinct points $\lambda_1,\dots,\lambda_r$, so
  $Q=\Psi$. Dividing by $p$, we record the identity
  \begin{equation}\label{eq:pe-necessary-transfer}
    p(z)\,\Xi(z)=\Psi(z),
    \qquad
    \Xi(z):=\alpha+\sum_{n=1}^{r}\frac{b_n}{z-\lambda_n}.
  \end{equation}

  \emph{Step 3: completion to $X$.} Take $W=\mathcal D(A)$ and $V=X$, set
  $\lambda_n:=-(n-r)$ for $r<n\le N$, and put
  \[
    A:=\operatorname{diag}(\lambda_n)_{n=1}^{N}.
  \]
  The $\lambda_n$ are distinct; the added ones are negative integers, whereas
  those from Step 2 lie in $(0,1)$. Then $A$ is self-adjoint with
  $\sup_n\lambda_n\le1$, so it generates an analytic semigroup. Being self-adjoint,
  $A^\star=A$ has the orthonormal eigenbasis $(e_n)_{n=1}^{N}$, so its system of
  root vectors is complete, and when $N=\infty$ it has compact resolvent, since
  $|\lambda_n|\to\infty$. Fix a unit vector $c\in U$, complete the coefficients of
  Step 2 by $b_n:=2^{-n}c$ for $n>r$, and define $B^\star e_n:=b_n$. All
  $b_n\neq0$ and $\sum_n\|b_n\|^2<\infty$, the tail being dominated by
  $\sum_n4^{-n}$. Hence these coordinates define a bounded operator
  $B^\star:X\to U$ (Hilbert--Schmidt when $N=\infty$); its adjoint
  $B:U\to X=V$ is bounded, and $B^\star$ is in particular graph-bounded.

  It remains to verify Assumption~\ref{hyp:H1}. Write
  $\Phi_Tu:=\int_0^TS(T-s)Bu(s)\,\d s$; its $n$-th eigencoordinate
  is $\int_0^Te^{\lambda_n(T-s)}\angl{b_n}{u(s)}_U\,\d s$. By
  Cauchy--Schwarz, this is bounded in modulus by
  \[
    \|b_n\|_U\,\|u\|_{L^2}
    \Bigl(\int_0^Te^{2\lambda_ns}\,\d s\Bigr)^{1/2},
    \qquad\text{where}\qquad
    \int_0^Te^{2\lambda_ns}\,\d s
    \le
    \begin{cases}
      Te^{2T}, & n\le r,\\[1mm]
      \bigl(2(n-r)\bigr)^{-1}, & n>r.
    \end{cases}
  \]
  Since $W=\mathcal D(A)$ and $2\in\rho(A)$, its graph norm is equivalent to
  $\|(2\operatorname{id}_X-A)\,\cdot\,\|_X$, and the two bounds above give
  \[
    \|(2\operatorname{id}_X-A)\Phi_Tu\|_{X}^2
    =\sum_{n=1}^{N}|2-\lambda_n|^2\,
     \Bigl|\int_0^Te^{\lambda_n(T-s)}\angl{b_n}{u(s)}_U\,\d s\Bigr|^2
    \le K\,\|u\|_{L^2}^2,
  \]
  with
  \[
    K:=4Te^{2T}\!\sum_{n\le r}\|b_n\|_U^2
      +\sum_{n>r}\frac{(2+n-r)^2}{2(n-r)}\,4^{-n}<\infty ,
  \]
  which verifies Assumption~\ref{hyp:H1}. Finally, each $\lambda_n$ is a simple eigenvalue with
  $B^\star e_n\neq0$, so
  $\ker B^\star\cap\ker(A-\lambda_n\operatorname{id}_X)=\{0\}$, and
  Proposition~\ref{prop:fattorini-hautus} yields approximate controllability.

  \emph{Step 4: the non-informative solution.} Set $\beta:=\sum_{n=1}^{r}e_n$;
  since $\alpha=\zeta_r\neq0$ we have $(\alpha,\beta)\neq0$. In terms of the
  operators just constructed, the function of
  \eqref{eq:pe-necessary-transfer} is
  $\Xi(z)=\alpha+B^\star(z\operatorname{id}_X-A)^{-1}\beta$.
  Let $x_n$ solve $\dot x_n=\lambda_n x_n+\angl{b_n}{\bar u}_U$ and set
  $\bar x:=\sum_{n=1}^{N}x_n e_n$. Then $\bar x$ is a solution of
  \eqref{eq:hilbert-control-formal} by Assumption~\ref{hyp:H1}. Since $\beta_n=0$ for
  $n>r$, only the first $r$ modes enter the combination
  \[
    w:=\angl{\alpha}{\bar u}_U+\angl{\beta}{\bar x}_X
      =\angl{\alpha}{\bar u}_U+\sum_{n=1}^{r}x_n.
  \]
  Applying $p(D)=\prod_{n\le r}(D-\lambda_n)$ and using
  $(D-\lambda_n)x_n=\angl{b_n}{\bar u}$,
  \[
    p(D)\,w=(p\Xi)(D)\,\bar u=\sum_{k=0}^{r}\angl{\zeta_k}{D^k\bar u}=0,
  \]
  so $w=\sum_{n=1}^{r}\kappa_n e^{\lambda_n t}$. Since the exponentials
  $e^{\lambda_n t}$ are linearly independent, a unique choice of the initial
  values $x_1(0),\dots,x_r(0)$, with $x_n(0)=0$ for $n>r$, yields
  $w\equiv0$; for
  $r=0$ this holds outright, the display above reading
  $w=\angl{\zeta_0}{\bar u}_U=0$. The nonzero pair $(\alpha,\beta)$ then
  satisfies $\angl{\alpha}{\bar u}_U+\angl{\beta}{\bar x}_X=0$ a.e.\ on $I$, so
  $(\alpha,\beta)$ is a nonzero element of $\ker\mathcal G(\bar u,\bar x)$ and
  $(\bar u,\bar x)$ is not informative.
\end{proof}

\subsection{Proof of Lemma~\ref{lem:window-pe-necessary} (excitation forced by window informativity)}
\label{app:window-pe-necessity}
We first show that the input windows are total in $L^2(0,T;U)$, that is,
\[
  \overline{\operatorname{span}
    \{\bar u(s+\cdot):0\le s\le T'\}}^{L^2(0,T;U)}
  =L^2(0,T;U).
\]
This follows by projecting
$\overline{\im\mathcal Z_{T,T'}^{\bar u,\bar x}}=L^2(0,T;U)\times X$ onto
its input component.  If excitation fails at order $L$, take
$\zeta_0,\ldots,\zeta_{L-1}$, not all zero, as in
Definition~\ref{def:PE}.  For $0\ne\psi\in\DD(0,T)$ set
$g=\sum_{k=0}^{L-1}(-1)^k\psi^{(k)}\zeta_k$.  Then $g\ne0$: otherwise
$\psi$ would solve a nontrivial constant-coefficient ODE, hence be analytic
and vanish identically, contrary to $\psi\ne0$.  Testing the
translated distributional relation against $\psi$ gives
\[
  \int_0^T\angl{\bar u(s+t)}{g(t)}_U\,\d t=0
  \qquad\text{for a.e.\ }s\in(0,T'),
\]
contradicting the totality of the input windows.
\qed

\subsection{Proof of Proposition~\ref{prop:finite-pe-sufficient} (finite-dimensional sufficiency)}
\label{app:finite-pe-sufficiency}
Let $(\alpha,\beta)\in U\times X$ annihilate the record:
\begin{equation}\label{eq:finite-suff-annihilator}
  \angl{\alpha}{\bar u}_U+\angl{\beta}{\bar x}_X=0
  \qquad\text{in }\DD'(I).
\end{equation}
For $k=0,\ldots,N$, put
$s_k:=\angl{(A^\star)^k\beta}{\bar x}_X$. The dynamics give
\[
  Ds_k=s_{k+1}+\angl{B^\star(A^\star)^k\beta}{\bar u}_U,
  \qquad k=0,\ldots,N-1.
\]
Starting from $s_0=-\angl{\alpha}{\bar u}_U$ and iterating this identity,
\begin{equation}\label{eq:finite-suff-recursion}
  s_k
  =-\sum_{j=0}^{k-1}
     \angl{B^\star(A^\star)^{k-1-j}\beta}{D^j\bar u}_U
   -\angl{\alpha}{D^k\bar u}_U ,
  \qquad k=0,\ldots,N .
\end{equation}
Let
$p(\lambda)=\lambda^N+a_{N-1}\lambda^{N-1}+\cdots+a_0$
be the characteristic polynomial of $A^\star$. Cayley--Hamilton gives
$s_N+\sum_{k=0}^{N-1}a_ks_k=0$. Substitution of
\eqref{eq:finite-suff-recursion} produces a relation
\[
  \sum_{j=0}^{N}\angl{\zeta_j}{D^j\bar u}_U=0
  \qquad\text{in }\DD'(I)
\]
for some $\zeta_0,\ldots,\zeta_N\in U$. Persistent excitation of
order $N+1$ makes every coefficient vanish. Reading the coefficients from
the highest derivative down gives successively
\[
  \alpha=0,\qquad
  B^\star\beta=B^\star A^\star\beta=\cdots
    =B^\star(A^\star)^{N-1}\beta=0 .
\]
Since $\dim X=N$, controllability of $(A,B)$ amounts to
\[
  \operatorname{span}\{A^kBv:\ v\in U,\ k=0,\ldots,N-1\}=X ,
\]
equivalently to $\bigcap_{k=0}^{N-1}\ker(B^\star(A^\star)^k)=\{0\}$,
so $\beta=0$. Thus the only annihilator in
\eqref{eq:finite-suff-annihilator} is the trivial one, and $(\bar u,\bar x)$
is informative.
\qed

\subsection{Proof of Proposition~\ref{prop:infinite-pe-insufficient} (infinite-order excitation counterexample)}
\label{app:infinite-pe-insufficient}
We use the notation of \eqref{eq:pe-counterexample}, writing
$Kx:=\angl{b}{x}_X$, so that $b\otimes b=BK$, and
\[
  \mathcal D(A_0):=\Big\{x\in X:\sum_{n\ge1}n^2
    |\angl{e_n}{x}_X|^2<\infty\Big\} .
\]
The operator $A_0$ is self-adjoint with
compact resolvent, while $BK$ is bounded and self-adjoint; hence the same is
true of $A$. Since $\|b\|_X^2=1/3$, for every $x\in\mathcal D(A)$,
\[
  \angl{x}{Ax}_X
  =-\sum_{n\ge1}n|\angl{e_n}{x}_X|^2+|\angl{b}{x}_X|^2
  \le-\tfrac23\|x\|_X^2 .
\]
Thus $A$ generates an analytic semigroup satisfying
$\|e^{tA}\|_{\mathcal L(X)}\le e^{-2t/3}$. Taking $W=V=X$, boundedness of
$B$ gives Assumption~\ref{hyp:H1} and makes $B^\star$ graph-bounded; moreover,
the root vectors of $A^\star=A$ are complete.

To verify approximate controllability, work on the complexified spaces and
suppose that $A^\star\eta=\lambda\eta$ and $B^\star\eta=0$. Then
$K\eta=0$ and $A_0\eta=A\eta-BK\eta=\lambda\eta$. If $\eta\ne0$, the simple
spectrum of $A_0$ gives $\eta=ce_n$ for some $c\ne0$ and $n\ge1$, contradicting
$0=B^\star\eta=c\,2^{-n}$. Proposition~\ref{prop:fattorini-hautus} now yields
approximate controllability of $(A,B)$.

Now, writing $z:=b$, define
\[
  \bar x(t):=e^{tA_0}z,\qquad
  \bar u(t):=-K\bar x(t)
    =-\sum_{n\ge1}4^{-n}e^{-nt}.
\]
Since $z\in\bigcap_{k\ge1}\mathcal D(A_0^k)$, both signals are smooth and
$\dot{\bar x}=A_0\bar x=A\bar x+B\bar u$, so $(\bar u,\bar x)$ is a
trajectory of $(A,B)$. Yet
$\bar u(t)+\angl{b}{\bar x(t)}_X=0$ on $I$, so the nonzero pair
$(1,b)\in U\times X$ annihilates the record. Thus the record is not
informative.

It remains to verify excitation. Fix $L\ge1$ and suppose that
$\sum_{j=0}^{L-1}\zeta_jD^j\bar u=0$ in $\DD'(I)$ for some
$\zeta_0,\ldots,\zeta_{L-1}\in\R$. The Weierstrass test and termwise
differentiation give, for every $j\ge0$ and uniformly on $I$,
\[
  D^j\bar u(t)=-\sum_{n\ge1}4^{-n}(-n)^j e^{-nt},
\]
since $\sum_{n\ge1}4^{-n}n^j<\infty$. Hence, for
$p(s):=\sum_{j=0}^{L-1}\zeta_js^j$, the distributional relation is equivalent
to the pointwise identity
\[
  \sum_{n\ge1}4^{-n}p(-n)e^{-nt}=0
  \qquad(t\in I).
\]
The function $F(q):=\sum_{n\ge1}4^{-n}p(-n)q^n$ is holomorphic on
$\{|q|<4\}$ and vanishes on $(e^{-T},1)$. Thus $F\equiv0$, so uniqueness of
power-series coefficients gives $p(-n)=0$ for every $n\ge1$. Consequently,
$p=0$ and $\zeta_0=\cdots=\zeta_{L-1}=0$. Since $L$ was arbitrary, $\bar u$
is persistently exciting of infinite order.
\qed

\subsection{Harmonically persistently exciting signals}
\label{app:harmonic-pe}

We first record the elementary properties of harmonic signals used in
Section~\ref{subsec:analytic-sufficient-input}. Let
$h(t)=\sum_{\ell\ge1}v_\ell e^{i\omega_\ell t}$ be a harmonic signal in the
sense of Definition~\ref{def:harmonic-signal}. The symmetry of the pairs
$(\omega_\ell,v_\ell)$ under $(\omega,v)\mapsto(-\omega,\overline v)$ makes $h$
real-valued, and absolute summability of the coefficients gives uniform
convergence of the series, so $h\in C_b(\R;U)$ and $h|_I\in L^2(I;U)$. The
coefficients are recovered, uniquely, by the Bohr mean
\[
  \hat h(\omega):=\lim_{S\to\infty}\frac1S\int_0^S e^{-i\omega t}h(t)\,\d t ,
\]
which equals $v_\ell$ at $\omega=\omega_\ell$ and vanishes at every other
frequency. Finally, since $|e^{i\omega_\ell t}|\le e^{\Omega|\operatorname{Im}t|}$
with $\Omega=\sup_\ell|\omega_\ell|<\infty$, the series converges locally
uniformly on $\CC$; thus $h$ is entire of exponential type at most $\Omega$,
and identities established on the finite window $I$ propagate to the whole
line.

\begin{proof}[Proof of Lemma~\ref{lem:harmonic-pe}]
  Since $2^{-j-k}\in(0,2^{-j})$, the frequency $\omega_{j,k}$ lies in the
  interval $I_j:=(2-2^{1-j},\,2-2^{-j})$, and the $I_j$ are pairwise disjoint
  subintervals of $(0,2)$; hence all the $\omega_{j,k}$ are pairwise distinct,
  and for each fixed $j$ the sequence $(\omega_{j,k})_{k\ge1}$ increases to
  $\nu_j:=2-2^{-j}$. Thus the first block starts at $\omega_{1,1}=\tfrac54$ and
  increases to $\nu_1=\tfrac32$, the second starts at
  $\omega_{2,1}=\tfrac{13}8$ and increases to $\nu_2=\tfrac74$, and so on. The
  complex exponential representation of $h(t)$ has
  frequencies $\pm\omega_{j,k}$, bounded by $\Omega\le2$, and coefficients
  $2^{-j-k-1}\phi_j$, which are absolutely summable because
  $\sum_{j\in J}\sum_{k\ge1}2^{-j-k}\le1$; it is therefore a harmonic signal, and it is
  harmonically persistently exciting with $a_{j,k}=2^{-j-k-1}$. If
  $\dim U=m<\infty$, the index set $J=\{1,\dots,m\}$ is finite and the
  construction is unchanged.
\end{proof}

\begin{proof}[Proof of Lemma~\ref{lem:harmonic-pe-implies-pe}]
  Suppose, for some $L\ge1$, that
  $\sum_{r=0}^{L-1}\angl{\zeta_r}{D^r\bar u}_U=0$ in $\DD'(I)$. The left-hand
  side extends to an entire scalar function. Since it vanishes distributionally
  on $I$, it vanishes pointwise there; the identity theorem then makes it
  identically zero on the whole real line.
  After division by $e^{\sigma t}$, its Bohr coefficient at the designated
  frequency $\omega_{j,k}$ is
  \[
    a_{j,k}\sum_{r=0}^{L-1}
      (\sigma+i\omega_{j,k})^r\angl{\zeta_r}{\phi_j}_U .
  \]
  Uniqueness of the Bohr coefficients makes this expression zero for every
  $k$. For fixed $j$ it is the value of a scalar polynomial of degree at most
  $L-1$ at the infinitely many distinct points $\sigma+i\omega_{j,k}$.
  The polynomial is therefore zero, so
  $\angl{\zeta_r}{\phi_j}_U=0$ for every $r$ and $j$. Totality of
  $(\phi_j)$ gives $\zeta_0=\cdots=\zeta_{L-1}=0$. Since $L$ was arbitrary,
  the excitation has infinite order.
\end{proof}

\subsection{Proof of Theorem~\ref{thm:analytic-universal-sufficiency} (universal sufficiency for analytic semigroups)}
\label{app:analytic-universal-sufficiency}

We work on the complexified spaces under the hypotheses of
Theorem~\ref{thm:analytic-universal-sufficiency}. Write $A_W$ for the generator
of $S|_W$ and pair $\eta\in W'$ with $W$ through
$\angl{\eta}{\cdot}_{W',W}$. Since $B$ may be unbounded in PDE applications,
the argument works in the frequency domain, through the \emph{transfer map}
\[
  \mathbf H(s):=(s\operatorname{id}-A_{X,-1})^{-1}B_X,
  \qquad s\in\rho(A),
\]
where $A_{X,-1}$ is the generator on $X_{-1}$, with
$\rho(A_{X,-1})=\rho(A)$. We first show that
Assumptions~\ref{hyp:H1} and~\ref{hyp:compat} make $\mathbf H$ holomorphic
with values in $\mathcal L(U,W)$ and give
$\|\mathbf H(r)\|_{\mathcal L(U,W)}=O(r^{-1/2})$ as $r\to+\infty$.

Since $\mathcal D(A)\hookrightarrow\mathcal D_V(A)\hookrightarrow W$,
$R(s,A):=(s\operatorname{id}-A)^{-1}\in\mathcal L(X,W)$ for $s\in\rho(A)$.
For $f\in W$, the vector $w=R(s,A)f$ satisfies $Aw=sw-f\in W$, so
$w\in\mathcal D(A_W)$. Thus $\rho(A)\subseteq\rho(A_W)$ and the resolvents
agree on $W$. By Appendix~\ref{app:control-realization}, $\mathbf H(r)$ is
$\mathcal L(U,W)$-valued for large real $r$, with the decay
\eqref{eq:control-resolvent-decay}. The identity
\[
  \mathbf H(s)=\mathbf H(r)-(s-r)R(s,A)\mathbf H(r)
\]
then makes $\mathbf H$ holomorphic into $\mathcal L(U,W)$ on all of $\rho(A)$.
Moreover, analyticity gives $S(1)\in\mathcal L(X,W)$; factoring
$S(t)|_W=S(1)S(t-1)|_W$ for $t\ge1$ and using local boundedness on $W$ gives
\begin{equation}\label{eq:analytic-W-growth-bound}
  \|S(t)\|_{\mathcal L(W)}\le M_W e^{\sigma t},\qquad t\ge0,
\end{equation}
for some $M_W\ge1$. Dualizing $S(\delta):X\to W$ gives
$S(\delta)^\star:W'\to X$. A further positive time maps $X$ into
$\mathcal D(A^\star)$ analytically. Factoring locally at a fixed positive time
shows that $r\mapsto S(r)^\star\eta$ is analytic into $\mathcal D(A^\star)$
on $(0,\infty)$ for every $\eta\in W'$.

Write $h(t)=\sum_{\ell\ge1}v_\ell e^{i\omega_\ell t}$ and set
\[
  s_\ell:=\sigma+i\omega_\ell,\qquad
  q_\ell:=\mathbf H(s_\ell)v_\ell,\qquad
  K:=\sigma+i\,\overline{\{\omega_\ell:\ell\ge1\}}.
\]
The compact set $K$ lies in the connected component $\Omega_\star$ of
$\rho(A)$ containing $\{\operatorname{Re}s>\sigma\}$: every point of $K$ has
a disc in $\rho(A)$ meeting that half-plane. Holomorphy gives
$C_K:=\sup_{s\in K}\|\mathbf H(s)\|_{\mathcal L(U,W)}<\infty$, hence
\begin{equation}\label{eq:analytic-qsummable}
  \sum_{\ell\ge1}\|q_\ell\|_W
  \le C_K\sum_{\ell\ge1}\|v_\ell\|_{U_{\CC}}<\infty.
\end{equation}
The finite-time Laplace identity on $X_{-1}$, applied to $B_Xv$, gives
\begin{equation}\label{eq:analytic-exp-convolution}
  L_t\bigl(e^{s\,\cdot}v\bigr)
  =e^{st}\mathbf H(s)v-S(t)\mathbf H(s)v,
  \qquad s\in\rho(A),\quad v\in U_{\CC},
\end{equation}
with both sides in $W$. The partial sums of $h$ converge in $L^2$ on bounded
intervals, so Assumption~\ref{hyp:H1} makes their control convolutions converge in
$C([0,\tau];W)$ for every $\tau>0$. Summing
\eqref{eq:analytic-exp-convolution} yields
\begin{equation}\label{eq:analytic-state-expansion}
  \bar x(t)=S(t)c+\sum_{\ell\ge1}q_\ell e^{s_\ell t},
  \qquad c:=x_0-\sum_{\ell\ge1}q_\ell,
\end{equation}
as an identity in $X$ for $x_0\in X$, and in $W$ for $x_0\in W$.

Both implications proved below start from an annihilator of the record, expand
it through \eqref{eq:analytic-state-expansion}, and eliminate the free response
by averaging. We isolate that step, which is the only one sensitive to whether
the state is read in $W$ or in $X$.

\begin{lem}[Vanishing of the harmonic coefficients]\label{lem:bohr-annihilator}
  Let $E$ be $W$ or $X$, with generator $A_E$ and growth bound
  \eqref{eq:analytic-W-growth-bound} or \eqref{eq:analytic-growth-bound}
  respectively. Let $c\in E_{\CC}$, let $\zeta\in E_{\CC}'$, and let
  $(d_\ell)_{\ell\ge1}$ be absolutely summable scalars such that
  \[
    \angl{\zeta}{S(t)c}_{E',E}
      +e^{\sigma t}\sum_{\ell\ge1}d_\ell e^{i\omega_\ell t}=0
  \]
  for a.e.\ $t$ in some nonempty open subinterval of $(0,\infty)$. Then
  $d_\ell=0$ for every $\ell$.
\end{lem}

\begin{proof}
  The harmonic term extends to an entire function, and $S(t)c$ is analytic into
  $\mathcal D(A)\hookrightarrow E$ for $t>0$, so the identity holds for every
  $t>0$. Multiply by $e^{-\sigma t}$ and take the Bohr mean at $\omega_\ell$:
  absolute summability permits termwise averaging, which returns $d_\ell$,
  while the free response averages away, since for
  $d:=(A_E-s_\ell\operatorname{id})^{-1}c$ the growth bound gives
  \[
    \frac1R\int_0^R\! e^{-s_\ell t}\angl{\zeta}{S(t)c}_{E',E}\,\d t
      =\frac{\angl{\zeta}{e^{-s_\ell R}S(R)d-d}_{E',E}}{R}
      \xrightarrow{R\to\infty}0 . \qedhere
  \]
\end{proof}

\begin{lem}[Transfer map as a Laplace transform]\label{lem:transfer-laplace}
  For $\eta\in W_{\CC}'$, the function
  $\psi_\eta(r):=B^\star S(r)^\star\eta$ is analytic on $(0,\infty)$ and
  belongs to $L^2_{\rm loc}([0,\infty);U_{\CC})$. Moreover,
  \begin{equation}\label{eq:analytic-transfer-laplace}
    \angl{\eta}{\mathbf H(s)v}_{W',W}
    =\int_0^\infty e^{-sr}\angl{\psi_\eta(r)}{v}_U\,\d r,
    \qquad \operatorname{Re}s>\sigma,\quad v\in U_{\CC}.
  \end{equation}
  If the left-hand side vanishes for every such $s$ and $v$, then
  $\psi_\eta\equiv0$ and $\eta$ annihilates $\im L_\tau$ for every $\tau>0$.
  If $(A,B)$ is approximately controllable in time $T$, then
  $\overline{\im L_T}^{\,W}=W$.
\end{lem}

\begin{proof}
  Analyticity follows from the smoothing just established and boundedness of
  $B^\star$ for the graph norm. Duality, first for controls supported away
  from the terminal time and then by Assumption~\ref{hyp:H1}, gives
  $(L_t^\star\eta)(\tau)=\psi_\eta(t-\tau)$ and local square integrability.
  Smoothing over a fixed positive time and \eqref{eq:analytic-W-growth-bound}
  also give $\|\psi_\eta(r)\|_U=O(e^{\sigma r})$ for $r\ge1$.
  Pairing \eqref{eq:analytic-exp-convolution} with $\eta$ yields
  \[
    \int_0^t e^{-sr}\angl{\psi_\eta(r)}{v}_U\,\d r
    =\angl{\eta}{\mathbf H(s)v}_{W',W}
       -e^{-st}\angl{\eta}{S(t)\mathbf H(s)v}_{W',W}.
  \]
  The last term tends to zero for $\operatorname{Re}s>\sigma$ by
  \eqref{eq:analytic-W-growth-bound}, proving the transform identity.
  Uniqueness of the Laplace transform and separability of $U$ show that a
  vanishing transform forces $\psi_\eta=0$, hence
  $\angl{\eta}{L_\tau u}_{W',W}=0$ for every $u$ and $\tau$.

  For the density assertion, let $\eta\in W'$ annihilate $\im L_T$.
  Then $\psi_\eta=0$ on $(0,T)$ and hence on $(0,\infty)$ by analyticity.
  For every $r>0$, the vector $\eta_r:=S(r)^\star\eta\in X$ satisfies
  $B^\star S(\rho)^\star\eta_r=\psi_\eta(r+\rho)=0$ for $\rho>0$.
  Thus $\eta_r\perp\im L_T$ in $X$, and approximate controllability gives
  $\eta_r=0$. Strong continuity of the adjoint semigroup on the Hilbert
  space $W'$ gives $\eta=\lim_{r\downarrow0}\eta_r=0$.
\end{proof}

\begin{proof}[Proof of Theorem~\ref{thm:analytic-universal-sufficiency}]
  \emph{\ref{item:analytic-approx-control}
  $\Rightarrow$\ref{item:analytic-all-informative}.}
  Fix $x_0\in W$ and let $(\alpha,\eta)\in U_{\CC}\times W_{\CC}'$
  annihilate the record on $I$. By \eqref{eq:analytic-state-expansion},
  \begin{equation}\label{eq:analytic-annihilator-expansion}
    0=\angl{\eta}{S(t)c}_{W',W}
      +e^{\sigma t}\sum_{\ell\ge1}d_\ell e^{i\omega_\ell t},
    \qquad
    d_\ell:=\angl{\alpha}{v_\ell}_U+\angl{\eta}{q_\ell}_{W',W},
  \end{equation}
  for a.e.\ $t\in I$, the coefficients being absolutely summable by
  \eqref{eq:analytic-qsummable}. Lemma~\ref{lem:bohr-annihilator} on $W$
  therefore gives $d_\ell=0$ for every $\ell$.

  Fix a direction $\phi_j$ of Definition~\ref{def:harmonic-pe}. The function
  \[
    f_j(s):=\angl{\alpha}{\phi_j}_U
             +\angl{\eta}{\mathbf H(s)\phi_j}_{W',W}
  \]
  is holomorphic on $\Omega_\star$ and vanishes at
  $s_{j,k}:=\sigma+i\omega_{j,k}$, because the corresponding coefficient is
  $a_{j,k}\phi_j$ with $a_{j,k}\ne0$. Since
  $s_{j,k}\to\sigma+i\nu_j\in K\subset\Omega_\star$, the identity theorem
  gives $f_j\equiv0$. Letting $s=r\to+\infty$ and using
  \eqref{eq:control-resolvent-decay} gives $\angl{\alpha}{\phi_j}_U=0$.
  Totality of $(\phi_j)$ implies $\alpha=0$ and then
  $\angl{\eta}{\mathbf H(s)v}_{W',W}=0$ for all $s\in\Omega_\star$ and
  $v\in U_{\CC}$. Lemma~\ref{lem:transfer-laplace} and approximate
  controllability give $\eta=0$.

  \emph{\ref{item:analytic-approx-control}
  $\Rightarrow$\ref{item:analytic-window-all-state}.}
  Fix $x_0\in X$ and let $(g,\xi)\in\ker(\mathcal Z_{T,T'}^{\bar u,\bar x})^\star$.
  Then
  \[
    \int_0^T\angl{g(t)}{\bar u(s+t)}_U\,\d t
       +\angl{\xi}{\bar x(s)}_X=0
    \qquad\text{for a.e.\ }s\in(0,T').
  \]
  Substituting \eqref{eq:analytic-state-expansion} and integrating termwise
  gives
  \[
    0=\angl{\xi}{S(s)c}_X+e^{\sigma s}\sum_{\ell\ge1}D_\ell e^{i\omega_\ell s},
    \qquad
    D_\ell:=\int_0^T e^{s_\ell t}\angl{g(t)}{v_\ell}_U\,\d t
               +\angl{\xi}{q_\ell}_X.
  \]
  The coefficients are absolutely summable, since the integral is bounded by
  $\sqrt T e^{|\sigma|T}\|g\|_{L^2}\|v_\ell\|_U$, so
  Lemma~\ref{lem:bohr-annihilator}, now on $X$, gives $D_\ell=0$ for every
  $\ell$. Consequently,
  \[
    F_j(z):=\int_0^T e^{zt}\angl{g(t)}{\phi_j}_U\,\d t
                +\angl{\xi}{\mathbf H(z)\phi_j}_X
  \]
  vanishes on $\Omega_\star$ by the same identity theorem. For
  $\operatorname{Re}z>\sigma$, Lemma~\ref{lem:transfer-laplace}, applied to
  $\xi\in X\hookrightarrow W'$, identifies $F_j(z)$ with
  $\int_{-T}^\infty e^{-zr}p_j(r)\,\d r$, where
  \[
    p_j(r):=
    \begin{cases}
      \angl{g(-r)}{\phi_j}_U,&-T<r<0,\\
      \angl{\psi_\xi(r)}{\phi_j}_U,&r\ge0.
    \end{cases}
  \]
  Translating by $T$ and using uniqueness of the Laplace transform gives
  $p_j=0$. Totality of $(\phi_j)$ gives $g=0$ and $\psi_\xi=0$, and
  approximate controllability then gives $\xi=0$.

  \emph{\ref{item:analytic-all-informative} or
  \ref{item:analytic-window-all-state} with $x_0=0$
  $\Rightarrow$\ref{item:analytic-approx-control}.}
  If $\overline{\im L_T}^{\,X}\ne X$, choose
  $0\ne\beta\in(\im L_T)^\perp$. Then $\psi_\beta=0$ on $(0,T)$ and
  hence on $(0,\infty)$ by analyticity. For the zero-state trajectory,
  \[
    \angl{\beta}{\bar x(s)}_X
    =\int_0^s\angl{\psi_\beta(s-r)}{\bar u(r)}_U\,\d r=0,
    \qquad s\ge0.
  \]
  Since $W$ is dense in $X$, $\beta$ also defines a nonzero functional on
  $W$. Thus $(0,\beta)$ contradicts either informativity in $U\times W$ or
  window informativity, proving both converse implications.
\end{proof}

\subsection{Proof of Example~\ref{ex:boundary-heat-informative} (boundary-controlled heat equation)}
\label{app:boundary-heat-informative}
\emph{The generator and the localized spaces.} Let
\[
  A=\partial_{\xi\xi},\qquad
  \mathcal D(A)=\{\varphi\in H^2(0,1):\varphi'(0)=\varphi'(1)=0\} ,
\]
which is self-adjoint with eigenvalues and normalized eigenfunctions
\[
  \lambda_0=0,\quad e_0(\xi)=1,\qquad
  \lambda_n=-n^2\pi^2,\quad e_n(\xi)=\sqrt2\cos(n\pi\xi),\qquad n\ge1 .
\]
It generates a bounded analytic semigroup with
$\|S(t)\|_{\mathcal L(X)}=1$, so \eqref{eq:analytic-growth-bound} holds with
$M=1$ and $\sigma=0$. Choose $\chi\in C_c^\infty(0,1)$ equal to one near
$\xi_0$ and define
\begin{equation}\label{eq:heat-localized-V}
  \|z\|_V^2:=\|z\|_{X_{-1}}^2+
    \int_0^\infty e^{-2t}\|\chi S(t)z\|_{H^1(0,1)}^2\,\d t,
\end{equation}
let $V$ be the space of those $z\in X_{-1}$ for which
\eqref{eq:heat-localized-V} is finite, and take $W=\mathcal D_V(A)$ with its
graph norm. Write $E=X_{-1}$, with norm
$\|z\|_E=\|(\operatorname{id}-A)^{-1}z\|_X$, use the same symbol $S$ for the
extended semigroup, and set $F_z(t):=e^{-t}\chi S(t)z$, so that
$\|z\|_V^2=\|z\|_E^2+\|F_z\|_{L^2(0,\infty;H^1)}^2$. For each $t>0$ the
operator $\chi S(t):E\to H^1(0,1)$ is bounded by analytic smoothing, so
$z\mapsto F_z$ is closed and $V$ is a Hilbert space. Multiplication by
$\chi$ is bounded on $H^1$, and the eigenfunction expansion gives
\[
  \int_0^\infty e^{-2t}\|S(t)z\|_{H^1(0,1)}^2\,\d t
    =\tfrac12\|z\|_X^2,\qquad z\in X,
\]
whence $X\hookrightarrow V$. From $F_{S(s)z}(t)=e^sF_z(t+s)$, strong
continuity of right translations in $L^2$ and of $S$ on $E$ make $S$
strongly continuous on $V$, with $\|S(s)\|_{\mathcal L(V)}\le e^s$; since
$S(s)z\in X$ for $s>0$, the same identity gives density. In particular,
$X\hookrightarrow V\hookrightarrow E$ continuously and densely.

Write $A_V$ for the generator on $V$, so that $W=\mathcal D(A_V)$ with
$\|w\|_W^2=\|w\|_V^2+\|A_Vw\|_V^2$ and Assumption~\ref{hyp:compat} holds by
construction, with $Z=W$. The consistent resolvent
$(2\operatorname{id}-A_V)^{-1}$ is the restriction of
$(2\operatorname{id}-A_{X,-1})^{-1}:E\to X$, so $W\hookrightarrow X$ and
\begin{equation}\label{eq:heat-W-domain}
  W=\{w\in X:A_{X,-1}w\in V\},\qquad A_Vw=A_{X,-1}w.
\end{equation}
Also $\mathcal D(A)\subset W$ densely in $X$, and $S$ restricts to a
strongly continuous semigroup on $W$.

\emph{Point observation and Assumption~\ref{hyp:H2}.}
For $w\in W$ the orbit satisfies $F_w\in H^1(0,\infty;H^1(0,1))$ with
$F_w'=F_{A_Vw}-F_w$, so its time trace $F_w(0)$, which equals $\chi w$
because $S(t)w\to w$ in $X$, obeys $\|\chi w\|_{H^1}\le c_\chi\|w\|_W$.
Combined with the embedding $H^1(0,1)\hookrightarrow C[0,1]$, of norm $c$,
this gives both $|Cw|=|w(\xi_0)|\le c\,c_\chi\|w\|_W$ and
\[
  \int_0^T|CS(t)w|^2\,\d t
    \le c^2 e^{2T}\int_0^\infty e^{-2t}
           \|\chi S(t)w\|_{H^1}^2\,\d t
    \le c^2e^{2T}\|w\|_V^2,
\]
which is Assumption~\ref{hyp:H2}. Point evaluation remains unbounded on
$X$, as smooth bumps near $\xi_0$, which lie in $\mathcal D(A)\subset W$,
show.

\emph{Boundary control and Assumption~\ref{hyp:H1}.}
Integration by parts gives $Bu=bu$, where $b=-\delta_0\in E$ has
coefficients $b_0=-1$ and $b_n=-\sqrt2$ for $n\ge1$. In particular,
$b\notin X$, so the control is genuinely unbounded on $X$. For $t>0$,
\[
  S(t)b=-K_N(t,\cdot,0),\qquad
  K_N(t,\xi,0)=\frac1{\sqrt{\pi t}}
    \sum_{m\in\mathbb Z}\exp\!\left(-\frac{(\xi-2m)^2}{4t}\right).
\]
Set $q(t):=\chi S(t)b$ and
$d_\chi:=\operatorname{dist}(\operatorname{supp}\chi,\{0\})>0$. Every
$\xi\in\operatorname{supp}\chi$ satisfies $|\xi-2m|\ge d_\chi$ for all
$m\in\mathbb Z$, so differentiating the image series gives, for suitable
constants $a,N$,
\[
  \|q(t)\|_{H^1}+\|q'(t)\|_{H^1}
    \le a\,t^{-N}e^{-d_\chi^2/(8t)},\qquad 0<t\le1,
\]
while for $t\ge1$ the eigenfunction expansion bounds $\|q(t)\|_{H^1}$ and,
the mode $\lambda_0=0$ being annihilated by the time derivative, makes
$\|q'(t)\|_{H^1}$ decay exponentially. Hence
\begin{equation}\label{eq:heat-separated-kernel}
  \int_0^\infty e^{-2t}
     \bigl(\|q(t)\|_{H^1}^2+\|q'(t)\|_{H^1}^2\bigr)\,\d t<\infty,
\end{equation}
and the definition of $V$ gives $b\in V$, hence $B\in\mathcal L(\R,V)$.

To prove the stronger control estimate into $W$, fix $u\in L^2(0,T)$
and put $w=L_Tu$. Cauchy--Schwarz in each Fourier coefficient gives
\begin{equation}\label{eq:heat-control-X}
  \|w\|_X^2\le\left(T+\sum_{n\ge1}\frac1{n^2\pi^2}\right)
    \|u\|_{L^2(0,T)}^2,
\end{equation}
so $a:=A_{X,-1}w\in E$ with $\|a\|_E\le c\|w\|_X$. For $s>0$, smoothing and
the convolution formula give $\chi S(s)a=\int_0^Tq'(s+T-r)u(r)\,\d r$, and
Cauchy--Schwarz in $r$ followed by the substitution $t=s+T-r$ yields
\[
  \int_0^\infty e^{-2s}\|\chi S(s)a\|_{H^1}^2\,\d s
    \le T e^{2T}\|u\|_{L^2(0,T)}^2
       \int_0^\infty e^{-2t}\|q'(t)\|_{H^1}^2\,\d t.
\]
With \eqref{eq:heat-separated-kernel} and \eqref{eq:heat-control-X} this
makes $a\in V$ with $\|a\|_V\le c_T\|u\|_{L^2(0,T)}$, so
\eqref{eq:heat-W-domain} gives
\[
  L_Tu\in W,\qquad \|L_Tu\|_W\le b_T\|u\|_{L^2(0,T)},
\]
proving Assumption~\ref{hyp:H1} for every $T>0$.

\emph{Excitation and informativity.}
The resolvent set is
$\rho(A)=\CC\setminus\{-n^2\pi^2:n\ge0\}$.
The frequencies $\pm\omega_k$ are distinct and bounded, lie in
$\pm(1,\tfrac32]$ and converge to $\pm1$, and the coefficients $2^{-k-1}$ of
the complex exponential representation are nonzero and absolutely summable.
Thus $\bar u$ is a harmonic signal, and it is harmonically persistently
exciting with the total family $\{1\}$ in $U_{\CC}=\CC$. Since
$\sigma=0$, the closure of the frequency set gives
$K=i\bigl(\{\pm\omega_k\}\cup\{\pm1\}\bigr)$, which consists of nonzero purely
imaginary numbers and therefore lies in $\rho(A)$, the spectrum being real.
Finally, $A$ has compact resolvent and a complete orthonormal eigenbasis,
and $B^\star\varphi=-\varphi(0)$ is bounded for the graph norm of
$\mathcal D(A)$. Each $\lambda_n$ is simple with $B^\star e_n=b_n\neq0$, so
Proposition~\ref{prop:fattorini-hautus} gives approximate controllability of
$(A,B)$, and Theorem~\ref{thm:analytic-universal-sufficiency} yields both
informativity in $U\times W$ for $\bar x_0\in W$ and window informativity
for $\bar x_0\in X$.
\qed

\subsection{Proof of Theorem~\ref{thm:willems-gramian} (Willems' fundamental lemma)}
\label{app:willems-gramian}
\emph{Step 1: regularity of the record.} By Assumption~\ref{hyp:H1} and the
discussion following it, $\bar x\in C(I;W)$ and
$\bar y=C\bar x\in C(I;Y)$. Moreover,
$B\bar u\in L^2(I;V_{-1})\subset L^1(I;V_{-1})$. Viewing the mild formula
\eqref{eq:hilbert-mild-state} on $V_{-1}$, Ball's variation-of-constants
theorem \parencite{ballStronglyContinuousSemigroups1977} shows that
$\bar x$ is a weak solution of $\dot x=Ax+B\bar u$ on $V_{-1}$. Let
$\phi\in\DD(I)$ and $z\in\mathcal D(A_{-1}^\star)$. Since
$\bar x(t)\in W\subset V=\mathcal D(A_{-1})$, the weak formulation gives
\[
  \Big\langle
    -\int_0^T\dot\phi(t)\bar x(t)\,\d t
    -\int_0^T\phi(t)\bigl(A\bar x(t)+B\bar u(t)\bigr)\,\d t
  ,z\Big\rangle_{V_{-1}}=0.
\]
The domain $\mathcal D(A_{-1}^\star)$ is dense in $V_{-1}$, so
$D\bar x=A\bar x+B\bar u$ in $\DD'(I;V_{-1})$. The right-hand side belongs
to $L^2(I;V_{-1})$, because $A\in\mathcal L(X,V_{-1})$,
$B\in\mathcal L(U,V_{-1})$, and
$\bar x\in C(I;W)\hookrightarrow L^2(I;X)$. Consequently,
\begin{equation}\label{eq:record-dynamics}
  \bar x\in H^1(I;V_{-1}),
  \qquad
  \dot{\bar x}=A\bar x+B\bar u
  \quad\text{in }L^2(I;V_{-1}).
\end{equation}

\emph{Step 2: identification of $\im\Gamma$.} Set $\bar z:=(\bar u,\bar x)$
and $\bar v:=(\bar u,\bar x,\dot{\bar x},\bar y)$. By
\eqref{eq:record-dynamics} and $\bar y=C\bar x$, the record satisfies
$\bar v=\Gamma\bar z$ almost everywhere, and informativity gives
$\cspan\bar z=U\times W$. Since $\Gamma$ is a homeomorphism onto its closed
range, the kernel identity for the Gramian
\eqref{eq:hilbert-gramian} yields
\[
  (\ker\mathcal G)^\perp
    =\cspan\bar v
    =\cspan(\Gamma\bar z)
    =\Gamma(\cspan\bar z)
    =\Gamma(U\times W)
    =\im\Gamma .
\]
For the synthesis operator, \eqref{eq:record-dynamics} and integration by
parts give
\begin{equation}\label{eq:synthesis-factorization}
  \mathcal W_{(\bar u,\bar x,\bar y)}\phi
  =\Gamma\left(
    \int_0^T\phi\,\bar u\,\d t,
    \int_0^T\phi\,\bar x\,\d t
  \right),
  \qquad \phi\in H_0^1(I),
\end{equation}
the operators $A$, $B$ and $C$ passing through the Bochner integrals because
they are bounded between the relevant spaces. Hence
$\im\mathcal W_{(\bar u,\bar x,\bar y)}=\Gamma(\mathcal N)$ for the subspace

\[
  \mathcal N:=\Bigl\{\Bigl(\int_0^T\phi\,\bar u\,\d t,
    \int_0^T\phi\,\bar x\,\d t\Bigr):\phi\in H_0^1(I)\Bigr\}
  \subseteq U\times W .
\]
If $(v,\xi)\in\mathcal N^\perp$, then
\[
  \int_0^T\phi(t)
  \bigl(
    \angl{v}{\bar u(t)}_U+\angl{\xi}{\bar x(t)}_W
  \bigr)\,\d t=0
  \qquad \forall\phi\in H_0^1(I),
\]
and since $H_0^1(I)$ is dense in $L^2(I)$ the scalar function in parentheses
vanishes almost everywhere; informativity then gives $(v,\xi)=0$. Thus
$\mathcal N$ is dense in $U\times W$ and, $\Gamma$ being a homeomorphism onto
its closed range,
$\overline{\im\mathcal W_{(\bar u,\bar x,\bar y)}}
=\Gamma(\overline{\mathcal N})=\Gamma(U\times W)=\im\Gamma$. This proves
\eqref{eq:graph-identification}.

\emph{Step 3: the equivalences.} Put $g:=Ax+Bu$, which lies in
$L^2(I;V_{-1})$ because $x\in C(I;W)\hookrightarrow L^2(I;X)$.

\ref{item:wfl-dynamics}$\Rightarrow$\ref{item:wfl-pointwise}: since
$\dot x=g$ and $y=Cx$ almost everywhere,
$(u,x,\dot x,y)(t)=\Gamma(u(t),x(t))\in\im\Gamma=(\ker\mathcal G)^\perp$ for a.e.\ $t$.

\ref{item:wfl-pointwise}$\Rightarrow$\ref{item:wfl-synthesis}:  $\mathcal W_{(u,x,y)}\phi$ is the Bochner
integral of $\phi(t)\,(u,x,\dot x,y)(t)$ over $I$. Its integrand takes values
in the closed subspace $(\ker\mathcal G)^\perp$ for a.e.\ $t$, hence so does the
integral.

\ref{item:wfl-synthesis}$\Rightarrow$\ref{item:wfl-dynamics}: fix
$\phi\in H_0^1(I)$. By \eqref{eq:graph-identification},
$\mathcal W_{(u,x,y)}\phi\in\im\Gamma$, and since the first two components of
$\Gamma(v,\xi)$ are $(v,\xi)$, its preimage must be
$(\int_0^T\phi\,u\,\d t,\int_0^T\phi\,x\,\d t)$. Comparing third components
gives
\[
  -\int_0^T\dot\phi\,x\,\d t
  =A\int_0^T\phi\,x\,\d t
    +B\int_0^T\phi\,u\,\d t
  =\int_0^T\phi\,g\,\d t ,
\]
so $Dx=g$ in $\DD'(I;V_{-1})$; since $x\in H^1(I;V_{-1})$ and
$g\in L^2(I;V_{-1})$, this is the identity $\dot x=g$ in $L^2(I;V_{-1})$.
Comparing fourth components gives
$\int_0^T\phi(t)\bigl(y(t)-Cx(t)\bigr)\,\d t=0$ for every $\phi\in H_0^1(I)$,
and pairing with arbitrary elements of $Y$ together with the density of
$H_0^1(I)$ in $L^2(I)$ shows that $y=Cx$ in $L^2(I;Y)$.

\emph{Step 4: the mild solution.} Whenever these conditions hold,
$x(t)\in W\subset V=\mathcal D(A_{-1})$ and $x\in H^1(I;V_{-1})$, so $x$ is a
strong solution on $V_{-1}$ and the variation-of-constants formula gives
$x(t)=S(t)x(0)+\int_0^tS(t-r)Bu(r)\,\d r$. By consistency of
the extrapolated semigroup and Assumption~\ref{hyp:H1}, this is precisely
\eqref{eq:hilbert-mild-state}; together with $y=Cx$ it follows that $(x,y)$
is the mild solution \eqref{eq:hilbert-mild} generated by $(x(0),u)$.
\qed

\subsection{Input--output parametrization and proof of Theorem~\ref{thm:windowed-io-fundamental-lemma}}
\label{app:windowed-io-fundamental-lemma}

\begin{lem}
  \label{lem:io-parametrization}
  Suppose that Assumptions~\ref{hyp:H1} and~\ref{hyp:H2} hold and let $T>0$.
  The \emph{parametrization}
  \[
    \begin{aligned}
      \mathcal M_T:L^2(0,T;U)\times X&\to L^2(0,T;U)\times L^2(0,T;Y),\\
      (u,x_0)&\mapsto\bigl(u,\mathcal F_Tu+\mathcal O_Tx_0\bigr),
    \end{aligned}
  \]
  is bounded, and $\im\mathcal M_T=\mathcal B_T^{u,y}$. If $(A,C)$ is approximately observable in
  time $T$, then $\mathcal M_T$ is injective. If $(A,C)$ is exactly observable
  in time $T$, then $\mathcal M_T$ is bounded below, hence a homeomorphism onto
  its closed range $\mathcal B_T^{u,y}$, and the latent state depends
  continuously on the pair, through
  \begin{equation}\label{eq:io-state-recovery}
    x_0=(\mathcal O_T^\star\mathcal O_T)^{-1}\mathcal O_T^\star
      (y-\mathcal F_Tu) .
  \end{equation}
  Moreover, if $(\bar u,\bar x,\bar y)$ is a record on
  $[0,T+T']$, then
  \begin{equation}\label{eq:window-factorization}
    \mathcal Y_{T,T'}^{\bar u,\bar y}
      =\mathcal M_T\,\mathcal Z_{T,T'}^{\bar u,\bar x} .
  \end{equation}
\end{lem}

\begin{proof}[Proof of Lemma~\ref{lem:io-parametrization}]
  Boundedness of $\mathcal M_T$ is that of $\mathcal F_T$ and $\mathcal O_T$
  in Lemma~\ref{lem:finite-horizon-output}, and
  $\im\mathcal M_T=\mathcal B_T^{u,y}$ is Definition~\ref{def:io-behavior}
  read off componentwise: $(u,y)\in\mathcal B_T^{u,y}$ if and only if
  $y=\mathcal O_Tx_0+\mathcal F_Tu$ for some $x_0\in X$, that is, if and only
  if $(u,y)=\mathcal M_T(u,x_0)$.

  If $\mathcal M_T(u,x_0)=0$, the first component gives $u=0$ and the second
  then gives $\mathcal O_Tx_0=0$; approximate observability makes
  $\mathcal O_T$ injective, so $x_0=0$. Under exact observability, write
  $(u,y):=\mathcal M_T(u,x_0)$; then $\|u\|\le\|(u,y)\|$ and
  $\mathcal O_Tx_0=y-\mathcal F_Tu$, so
  \[
    c_T\|x_0\|_X
    \le\|\mathcal O_Tx_0\|_{L^2(0,T;Y)}
    \le\|y\|_{L^2(0,T;Y)}+\|\mathcal F_T\|\,\|u\|_{L^2(0,T;U)}
    \le\bigl(1+\|\mathcal F_T\|\bigr)\|(u,y)\| ,
  \]
  which bounds $\|(u,x_0)\|$ by a multiple of $\|\mathcal M_T(u,x_0)\|$. A
  bounded-below operator on a Hilbert space has closed range and is a
  homeomorphism onto it; exact observability also makes
  $\mathcal O_T^\star\mathcal O_T\ge c_T^2$ boundedly invertible, whence
  \eqref{eq:io-state-recovery}.

  Finally, both window synthesis operators integrate against the same
  $\phi\in L^2(0,T')$, and
  $(\bar u(s+\cdot),\bar y(s+\cdot))=\mathcal M_T(\bar u(s+\cdot),\bar x(s))$ for almost every
  $s\in(0,T')$ by the shift identity of
  Lemma~\ref{lem:finite-horizon-output}. Since $\mathcal M_T$ is linear and
  bounded, it passes through the Bochner integral, so
  \[
    \mathcal Y_{T,T'}^{\bar u,\bar y}\phi
    =\int_0^{T'}\phi(s)\,\mathcal M_T(\bar u(s+\cdot),\bar x(s))\,\d s
    =\mathcal M_T\int_0^{T'}\phi(s)(\bar u(s+\cdot),\bar x(s))\,\d s
    =\mathcal M_T\,\mathcal Z_{T,T'}^{\bar u,\bar x}\phi . \qedhere
  \]
\end{proof}

\begin{proof}[Proof of Theorem~\ref{thm:windowed-io-fundamental-lemma}]
  By the factorization \eqref{eq:window-factorization} of
  Lemma~\ref{lem:io-parametrization},
  \[
    \im\mathcal Y_{T,T'}^{\bar u,\bar y}
    =\mathcal M_T\bigl(\im\mathcal Z_{T,T'}^{\bar u,\bar x}\bigr)
    \subseteq\im\mathcal M_T
    =\mathcal B_T^{u,y},
  \]
  so $\overline{\im\mathcal Y_{T,T'}^{\bar u,\bar y}}
  \subseteq\overline{\mathcal B_T^{u,y}}$. Conversely, $\mathcal M_T$ is
  bounded, hence
  \[
    \mathcal M_T\Bigl(
      \overline{\im\mathcal Z_{T,T'}^{\bar u,\bar x}}\Bigr)
    \subseteq
    \overline{\mathcal M_T\bigl(\im\mathcal Z_{T,T'}^{\bar u,\bar x}\bigr)}
    =\overline{\im\mathcal Y_{T,T'}^{\bar u,\bar y}} .
  \]
  Window informativity gives
  $\overline{\im\mathcal Z_{T,T'}^{\bar u,\bar x}}=L^2(0,T;U)\times X$, so
  the left-hand side is $\im\mathcal M_T=\mathcal B_T^{u,y}$ and therefore
  $\overline{\mathcal B_T^{u,y}}
  \subseteq\overline{\im\mathcal Y_{T,T'}^{\bar u,\bar y}}$. The two
  inclusions give
  $\overline{\im\mathcal Y_{T,T'}^{\bar u,\bar y}}
  =\overline{\mathcal B_T^{u,y}}$, which is
  \eqref{eq:io-behavior-identification}.
\end{proof}

\subsection{Proof of Theorem~\ref{thm:data-fattorini-hautus} (data tests for approximate controllability)}
\label{app:data-fattorini-hautus}
We work in the complexified spaces. For $\eta\in\mathcal D(A^\star)_{\CC}$,
set $y_\eta(t):=\angl{\eta}{\bar x(t)}_X$. By Assumption~\ref{hyp:H1}, the
graph-boundedness of $B^\star$, and the weak formulation of the mild solution
\parencite{ballStronglyContinuousSemigroups1977}, $y_\eta\in H^1(I;\CC)$ and
\begin{equation}\label{eq:weak-mild-diff}
  \dot y_\eta(t)
    =\angl{A^\star\eta}{\bar x(t)}_X
     +\angl{B^\star\eta}{\bar u(t)}_U
  \qquad\text{for a.e.\ }t\in I.
\end{equation}

\emph{Step 1: obstructions produce exponential identities.} Suppose that
$A^\star\eta=\lambda\eta$ and $B^\star\eta=0$ for some $\lambda\in\CC$ and
$\eta\in\mathcal D(A^\star)_{\CC}\setminus\{0\}$. Then
\eqref{eq:weak-mild-diff} gives $\dot y_\eta=\lambda y_\eta$, and hence
\begin{equation}\label{eq:data-fh-exponential}
  \angl{\eta}{\bar x(t)}_X=\angl{\eta}{x_0}_X\,e^{\lambda t}
  \qquad(t\in I).
\end{equation}
Informativity plays no role here. If $(A,B)$ is not approximately
controllable, Proposition~\ref{prop:fattorini-hautus} provides such a pair
$(\lambda,\eta)$, and \eqref{eq:data-fh-exponential} is \eqref{eq:data-fh}
with $\kappa=\angl{\eta}{x_0}_X$. Since
$\mathcal D(A^\star)_{\CC}\subset X_{\CC}$, contraposition proves
\ref{item:data-fh-certificate} and the ``if'' part of
\ref{item:data-fh-exact}.

\emph{Step 2: informativity turns them into obstructions.} Recall from
Definition~\ref{def:hilbert-informative} that informativity rules out
annihilators $(\alpha,\beta)\in U_{\CC}\times X_{\CC}$ of the record as
well. Suppose that
\eqref{eq:data-fh} holds for some $\lambda,\kappa\in\CC$ and
$\eta\in\mathcal D(A^\star)_{\CC}\setminus\{0\}$. Differentiating it and
using the complex-bilinear convention for the pairing gives
$\dot y_\eta=\lambda y_\eta=\angl{\lambda\eta}{\bar x}_X$ a.e.\@ Subtracting
\eqref{eq:weak-mild-diff} therefore gives
\[
  \angl{A^\star\eta-\lambda\eta}{\bar x(t)}_X
  +\angl{B^\star\eta}{\bar u(t)}_U=0
  \qquad\text{for a.e.\ }t\in I.
\]
Applying informativity with
$\alpha:=B^\star\eta\in U_{\CC}$ and
$\beta:=A^\star\eta-\lambda\eta\in X_{\CC}$ yields
\[
  B^\star\eta=0,
  \qquad
  A^\star\eta=\lambda\eta.
\]
By Proposition~\ref{prop:fattorini-hautus}, $(A,B)$ is not approximately
controllable, which is the ``only if'' part of \ref{item:data-fh-exact}.

\emph{Step 3: the constant.} Evaluating \eqref{eq:data-fh} at $t=0$ gives
$\kappa=\angl{\eta}{x_0}_X$. If $\kappa=0$, then $(0,\eta)\ne0$ annihilates
the record, contrary to informativity. Hence
$\kappa\ne0$.
\qed

\subsection{Proof of Theorem~\ref{thm:io-window-controllability} (input--output controllability tests)}
\label{app:io-window-controllability}
We work in the complexified spaces and write
\[
  \ell_{v,g}(u,y)
    :=\int_0^T\angl{v(t)}{u(t)}_U\,\d t
     +\int_0^T\angl{g(t)}{y(t)}_Y\,\d t
\]
for the bounded window functional on the left-hand side of
\eqref{eq:io-window-fh}, so that the identity reads
$\ell_{v,g}(\bar u(s+\cdot),\bar y(s+\cdot))=\kappa e^{\lambda s}$. Pulling $\ell_{v,g}$ back through
the parametrization $\mathcal M_T$ of Lemma~\ref{lem:io-parametrization} gives
\begin{equation}\label{eq:io-window-pullback}
  \ell_{v,g}\bigl(\mathcal M_T(u,x)\bigr)
    =\angl{\alpha}{u}_{L^2(0,T;U)}+\angl{\eta}{x}_X,
  \quad
  \alpha:=v+\mathcal F_T^\star g,
  \quad \eta:=\mathcal O_T^\star g .
\end{equation}
Moreover, if $\bar x$ denotes the state trajectory of the record, then
Lemma~\ref{lem:finite-horizon-output} gives the shift identity
$(\bar u(s+\cdot),\bar y(s+\cdot))=\mathcal M_T(\bar u(s+\cdot),\bar x(s))$ for every
admissible $s$.

\emph{Step 1: obstructions produce exponential identities.} Suppose that
$(A,B)$ is not approximately controllable, and use
Proposition~\ref{prop:fattorini-hautus} to choose $\lambda\in\CC$ and
$\eta\in\mathcal D(A^\star)_{\CC}\setminus\{0\}$ with
$A^\star\eta=\lambda\eta$ and $B^\star\eta=0$. Exact observability in time
$T$ makes $\mathcal O_T^\star$ onto, so there is $g\in L^2(0,T;Y_{\CC})$
with $\mathcal O_T^\star g=\eta$, and $g\ne0$ because $\eta\ne0$. Choosing
$v:=-\mathcal F_T^\star g$ makes $\alpha=0$ in
\eqref{eq:io-window-pullback}, so the shift identity and
\eqref{eq:data-fh-exponential} yield
\begin{equation}\label{eq:io-window-projection}
  \ell_{v,g}(\bar u(s+\cdot),\bar y(s+\cdot))
    =\angl{\eta}{\bar x(s)}_X
    =\angl{\eta}{x_0}_X\,e^{\lambda s}
\end{equation}
for every admissible $s$. Thus \eqref{eq:io-window-fh} holds with
$\kappa:=\angl{\eta}{x_0}_X$ and $g\ne0$. Only the record and the shift
identity are used; Step 2 discharges the remaining requirement $\kappa\ne0$.

\emph{Step 2: informativity forces $\kappa\ne0$.} Let $(\lambda,\kappa,v,g)$
be as in Step 1 and suppose that $\kappa=\angl{\eta}{x_0}_X=0$. Since
$A^\star\eta=\lambda\eta$ and $B^\star\eta=0$, identity
\eqref{eq:data-fh-exponential} makes $\angl{\eta}{\bar x(t)}_X$ vanish for
every $t$ in the record interval, so $(0,\eta)\ne0$ annihilates
$(\bar u,\bar x)$ there. Under the hypothesis of
\ref{item:io-fh-certificate} this contradicts informativity on $(0,T+T')$
outright; under the hypothesis of \ref{item:io-fh-exact} it contradicts
informativity on $(0,3T+T')$, which window informativity at horizon $3T$
supplies through Lemma~\ref{lem:window-implies-informative}. Hence
$\kappa\ne0$, and contraposition of Step 1 proves
\ref{item:io-fh-certificate} and the ``if'' part of
\ref{item:io-fh-exact}.

\emph{Step 3: exponential identities produce obstructions.} Here
$(\bar u,\bar x)$ is window informative at horizon $3T$ on $[0,3T+T']$, so
that Theorem~\ref{thm:windowed-io-fundamental-lemma} gives
\begin{equation}\label{eq:io-window-behavior}
  \overline{\im\mathcal Y_{3T,T'}^{\bar u,\bar y}}
    =\overline{\mathcal B_{3T}^{u,y}} .
\end{equation}
Suppose that \eqref{eq:io-window-fh} holds for every $s\in(0,2T+T')$, with
$\kappa\ne0$. Write $u_r:=u(r+\cdot)|_{(0,T)}$ and
$y_r:=y(r+\cdot)|_{(0,T)}$. For $r\in(T,2T)$, the functional
\[
  D_r(u,y):=\ell_{v,g}(u_r,y_r)-e^{\lambda r}\ell_{v,g}(u_0,y_0)
\]
annihilates every measured length-$3T$ window. Identity
\eqref{eq:io-window-behavior} and continuity therefore imply that $D_r$ annihilates
$\mathcal B_{3T}^{u,y}$, so every length-$3T$ trajectory satisfies
\begin{equation}\label{eq:io-window-shift-eigenrelation}
  \ell_{v,g}(u_r,y_r)=e^{\lambda r}\ell_{v,g}(u_0,y_0),
  \qquad r\in(T,2T).
\end{equation}
In \eqref{eq:io-window-shift-eigenrelation}, first choose zero initial state
and an arbitrary input supported in $(r,r+T)$, then choose zero input and an
arbitrary initial state, and finally choose zero initial state and arbitrary
inputs supported in $(0,r)$. With $\alpha$ and $\eta$ as in
\eqref{eq:io-window-pullback}, these three choices yield, respectively,
\[
  \alpha=0,\qquad
  S(r)^\star\eta=e^{\lambda r}\eta,\qquad
  L_r^\star\eta=0,
  \qquad r\in(T,2T).
\]
Moreover, $\eta\ne0$; otherwise \eqref{eq:io-window-pullback} and
$\alpha=0$ would make $\ell_{v,g}$ vanish on every admissible window,
contrary to \eqref{eq:io-window-fh} and $\kappa\ne0$. In particular
$g\ne0$, which is the last assertion of \ref{item:io-fh-exact}.

For $0<\tau<T$, choose $r\in(T,2T-\tau)$. Applying $S(\tau)^\star$
to the identity at $r$ and comparing with the identity at $r+\tau$
gives $S(\tau)^\star\eta=e^{\lambda\tau}\eta$. The semigroup property
extends this to every $\tau\ge0$. Hence
$\eta\in\mathcal D(A^\star)_{\CC}$ and
$A^\star\eta=\lambda\eta$. Since $L_r^\star\eta=0$, duality gives, for
arbitrary $u$,
\[
  0=\angl{\eta}{L_ru}_X
    =\int_0^r e^{\lambda(r-t)}
      \angl{B^\star\eta}{u(t)}_U\,\d t,
\]
and thus $B^\star\eta=0$. Proposition~\ref{prop:fattorini-hautus} shows that
$(A,B)$ is not approximately controllable, which is the ``only if'' part of
\ref{item:io-fh-exact}.
\qed

\subsection{Proof of Theorem~\ref{thm:lqr-data} (data-driven LQR)}
\label{app:lqr-data}
Define the affine control-to-trajectory map
\[
  \Lambda u:=
  \bigl(u,x(\cdot\,;x_0,u),Cx(\cdot\,;x_0,u)\bigr).
\]
Assumption~\ref{hyp:H1} and the regularity argument in the proof of
Theorem~\ref{thm:willems-gramian} give
\(x(\cdot\,;x_0,u)\in C(I;W)\cap H^1(I;V_{-1})\), and
\[
  \mathcal J(\Lambda u)=J(u;x_0).
\]
The equivalence \ref{item:wfl-dynamics}$\Leftrightarrow$\ref{item:wfl-pointwise}
of Theorem~\ref{thm:willems-gramian} shows that
\(\Lambda\) maps \(L^2(I;U)\) bijectively onto
\(\mathcal T_{\mathcal G}(x_0)\): every mild trajectory satisfies the pointwise graph
constraint, and every triple satisfying that constraint is the mild
trajectory generated by its input and the prescribed initial state.

Hence the data-driven and model-based problems have the same feasible
trajectories and objective values. The model-based cost is a continuous
quadratic functional of \(u\); its quadratic part is coercive because
\[
  J(u;0)\ge
  \int_0^T\angl{u(t)}{Ru(t)}_U\,\d t
  \ge\eps_R\|u\|_{L^2(I;U)}^2.
\]
It therefore has a unique minimizer, which proves the first assertion. Under
Assumptions~\ref{hyp:H2} and~\ref{hyp:compat}, the feedback formula and value
follow from Theorem~\ref{thm:lqr-riccati}.
\qed

\subsection{Proof of Theorem~\ref{thm:lqr-data-io} (input--output data-driven LQR)}
\label{app:lqr-data-io}
Write $\rho:=T_0+T$ and let $\mathcal M_\rho$ be the parametrization of
Lemma~\ref{lem:io-parametrization} at that length.

\emph{Step 1: the feasible set is the behavior.} Since
$(\mathcal O_\rho x)\vert_{(0,T_0)}=\mathcal O_{T_0}x$ for every $x\in X$,
exact observability in time $T_0$ implies exact observability in time $\rho$
with the same constant. Together with Assumption~\ref{hyp:H2}, this makes
$\mathcal M_\rho$ bounded below, hence injective with closed range
$\mathcal B_\rho^{u,y}$ (Lemma~\ref{lem:io-parametrization}). Window
informativity at horizon $\rho$ and
Theorem~\ref{thm:windowed-io-fundamental-lemma} then give
\[
  \overline{\im\mathcal Y_{\rho,T'}^{\bar u,\bar y}}
  =\overline{\mathcal B_\rho^{u,y}}
  =\mathcal B_\rho^{u,y} ,
\]
so the feasible set of \eqref{eq:lqr-data-io-problem} is
\[
  \mathcal A=\bigl\{\mathcal M_\rho(v,\xi)\ :\
    (v,\xi)\in L^2(0,\rho;U)\times X,\quad
    \mathcal M_\rho(v,\xi)\vert_{(0,T_0)}=(u_{\rm ini},y_{\rm ini})\bigr\} .
\]

\emph{Step 2: the conditioning window fixes the state.} Because
$(\mathcal F_\rho v)(t)=CL_tv$ depends only on $v\vert_{(0,t)}$, restriction
to $(0,T_0)$ gives
$\mathcal M_\rho(v,\xi)\vert_{(0,T_0)}=\mathcal M_{T_0}(v\vert_{(0,T_0)},\xi)$.
Hence $\mathcal M_\rho(v,\xi)\in\mathcal A$ if and only if
$v\vert_{(0,T_0)}=u_{\rm ini}$ and
$\mathcal O_{T_0}\xi=\mathcal O_{T_0}x_{\rm ini}$, the latter because
$(u_{\rm ini},y_{\rm ini})=\mathcal M_{T_0}(u_{\rm ini},x_{\rm ini})$ and the
forced responses cancel. Exact observability in time $T_0$ makes
$\mathcal O_{T_0}$ injective, so $\xi=x_{\rm ini}$. Consequently
$\mathcal A\neq\emptyset$ and, by the injectivity of $\mathcal M_\rho$, the
map $v\mapsto\mathcal M_\rho(v,x_{\rm ini})$ is a bijection from
$\{v\in L^2(0,\rho;U):v\vert_{(0,T_0)}=u_{\rm ini}\}$ onto $\mathcal A$. Since the
conditioning part of $v$ is prescribed, that set is itself parametrized
bijectively by the free part $v(T_0+\cdot)\in L^2(0,T;U)$.

\emph{Step 3: the cost is the model-based cost.} Let
$(v,z)=\mathcal M_\rho(v,x_{\rm ini})\in\mathcal A$ and let $x$ be the mild
solution generated by $(x_{\rm ini},v)$. Then
$x(T_0)=x(T_0;x_{\rm ini},u_{\rm ini})=x_0$ by causality, and $x_0\in W$
because $x_{\rm ini}\in W$, by \eqref{eq:hilbert-mild}. The shift identity of
Lemma~\ref{lem:finite-horizon-output} gives
\[
  z(T_0+\cdot)=\mathcal O_Tx_0+\mathcal F_Tv(T_0+\cdot) ,
\]
that is, the regulated part of the window is the output generated on $[0,T]$
by the initial state $x_0$ and the input $v(T_0+\cdot)$. Since $G=0$, comparing
the objective of \eqref{eq:lqr-data-io-problem} with \eqref{eq:lqr-cost}
yields that the cost of $(v,z)$ equals $J(v(T_0+\cdot);x_0)$.

\emph{Step 4: conclusion.} By Steps 2 and 3, problem
\eqref{eq:lqr-data-io-problem} is the model-based problem \eqref{eq:lqr-cost}
for the initial state $x_0$, transported by the bijection
$(v,z)\mapsto v(T_0+\cdot)$. As in the proof of Theorem~\ref{thm:lqr-data}, the
functional $J(\cdot\,;x_0)$ is continuous and quadratic with coercive
quadratic part, so it has a unique minimizer $u_F$. Pulling $u_F$ back through
the bijection gives the unique minimizer $(u^\star,y^\star)$ of
\eqref{eq:lqr-data-io-problem}, namely the input $u^\star$ with
$u^\star\vert_{(0,T_0)}=u_{\rm ini}$ and $u^\star(T_0+\cdot)=u_F$, together with
$y^\star=\mathcal O_\rho x_{\rm ini}+\mathcal F_\rho u^\star$; its regulated
part is the input--output pair generated by $(x_0,u_F)$. Finally, since
$x_0\in W$, Assumptions~\ref{hyp:H1}, \ref{hyp:H2}, and~\ref{hyp:compat}
and Theorem~\ref{thm:lqr-riccati} evaluate the common optimal value as
$J(u_F;x_0)=\angl{x_0}{P(0)x_0}_{V,V'}$.
\qed

\clearpage
\begingroup
\setlength{\emergencystretch}{1em}
\printbibliography
\endgroup

\end{document}